\documentclass[12pt, a4paper, final]{amsart}
\allowdisplaybreaks
\usepackage{amssymb, amsmath, amsthm, amsbsy} 
\usepackage{mathtools}
\usepackage[shortlabels]{enumitem}
\usepackage[margin=1.1in]{geometry}
\usepackage[ddmmyyyy,hhmmss]{datetime}
\usepackage{natbib} 
\usepackage[colorlinks=true,
  urlcolor=blue,
  linkcolor=blue,
  citecolor=blue,
  pdfencoding=auto]{hyperref}
\usepackage{xcolor}
\newif\ifmtpro 
\ifmtpro

  \usepackage[scaled=0.85]{helvet}
  \usepackage[subscriptcorrection]{mtpro}
  \usepackage[mtpscr, mtpfrak]{mtpb} 
\fi
\usepackage{microtype}
\newtheorem{lemma}{Lemma}
 
\newtheorem{theorem}{Theorem}

\theoremstyle{definition}
\newtheorem{definition}{Definition}

\newtheorem{remark}{Remark}

\usepackage{accents}
\usepackage{tikz}
\makeatletter
  \newcommand*\bigcdot{\mathpalette\bigcdot@{1}}
  \newcommand*\bigcdot@[2]{
    \mathbin{\hbox{\scalebox{#2}{$\m@th#1\begin{tikzpicture}
            \draw[line width=.7pt] (0,0) circle
            (1pt);\end{tikzpicture}$}}}
  }
\makeatother

\newcommand{\E}{\mathbb{E}}

\newcommand{\Prb}{\mathbb P}
\newcommand{\R}{\mathbb R}
\newcommand{\Z}{\mathbb Z}
\newcommand{\N}{\mathbb N}
\newcommand{\Ind}{{\bf 1}}
\newcommand{\Eq}{F^{-1}(0)}
\newcommand{\sL}{\mathfrak{L}}

\newcommand{\D}{\mathfrak D}
\newcommand{\TsD}{\mathrm{T}\D} 
\newcommand{\iD}{\accentset{\bigcdot}{\D}}

\newcommand{\Ent}{\mathrm{Ent}}
\newcommand{\JacF}{\mathit{JF}}
\ifmtpro
 \newcommand{\CR}{\mathscr R(\Phi)}
 \newcommand{\Bhood}{\mathscr B}
 \newcommand{\F}{\mathscr F}
 \newcommand{\undd}[1]{\underset{\wcheck{}}{#1}}
\else
 \newcommand{\CR}{\mathcal R(\Phi)}
 \newcommand{\Bhood}{\mathcal B}
 \newcommand{\F}{\mathfrak F}
 \newcommand{\undd}[1]{\underset{\widehat{}}{#1}}
\fi
\begin{document}

\title[Interacting Vertex Reinforced Random Walks]{Vertex reinforced
  random walks with exponential interaction on complete graphs}
\author[B. Pires]{Benito Pires}
\thanks{The first author was supported by grant \#2019/10269-3 S\~ao
  Paulo Research Foundation (FAPESP)}
\author[F. P. A. Prado]{Fernando P. A. Prado}
\address[B. Pires, F. P. A. Prado and R. A. Rosales]{Departamento de
  Computa\c{c}\~ao e Matem\'atica, Faculdade de Filosofia, Ci\^encias
  e Letras, Universidade de S\~ao Paulo, 
   14040-901, Ribeir\~ao Preto - SP, Brazil}
\email{benito@usp.br, prado1@usp.br, rrosales@usp.br} 
\author[R. A. Rosales]{Rafael A. Rosales}

\subjclass[2010]{Primary 60K35; secondary 37C10, 60J10, 60G50}
\keywords{reinforced random walk, stochastic approximation,
  stability}
\date{28/12/2020}

\maketitle

\begin{abstract}
We describe a model for $m$ vertex reinforced interacting random walks
on complete graphs with $d\geq 2$ vertices. The transition probability
of a random walk to a given vertex depends exponentially on the
proportion of visits made by all walks to that vertex. The individual
proportion of visits is modulated by a strength parameter that can be
set equal to any real number. This model covers a large variety of
interactions including different vertex repulsion and attraction
strengths between any two random walks as well as self-reinforced
interactions. We show that the process of empirical vertex occupation
measures defined by the interacting random walks converges (a.s.) to
the limit set of the flow induced by a smooth vector field. Further,
if the set of equilibria of the field is formed by isolated points,
then the vertex occupation measures converge (a.s.) to an equilibrium
of the field. These facts are shown by means of the construction of a
strict Lyapunov function. We show that if the absolute value of
the interaction strength parameters are smaller than a certain upper 
bound, then, for any number of random walks ($m\geq 2$) on any graph
($d \geq 2$), the vertex occupation measure converges toward a unique
equilibrium. We provide two additional examples of repelling random
walks for the cases $m=d=2$ and $m=3$, $d=2$. The latter is used to
study some properties of three exponentially repelling random walks on
$\Z$.
\end{abstract}


\section{Introduction}\label{sec:intro}
Let $G=(E,V)$ be a finite complete graph with $d\geq 2$ vertices and
let $W = \{(W^1(n)$, $\ldots$, $W^m(n))\}_{n \geq 1}$ be a process
described by $m \geq 2$ interacting random walks on $G$. The process
$W$, defined on a suitable probability space $(\Omega, \F, \Prb)$, is
constructed as follows. For each $v \in [d] = \{1, 2, \ldots, d\}$ and
$i \in [m]$, let $X^i_v(0) = 1$, and then for $n\geq 1$ let $X_v^i(n)$
be the empirical occupation measure of the vertices by the $i$-th
walk, that is,
\begin{equation}
  \label{eqn:occupation_measure}
  X_v^i(n) = \frac{1}{d+n}\Big(1+\sum_{k=1}^n
  \Ind{\big\{W^i(k) = v\big\}}\Big). 
\end{equation}
Define $\F_0 = \{\Omega, \emptyset\}$ and $\F_n = \sigma(W(k): 1 \leq
k \leq n)$ as the filtration generated by $W$ up to time $n \geq
1$. Then, for all $v \in [d]$, $i, j \in [m]$ and $n\geq 0$
define the transition probability of $W$ as
\begin{equation}
\label{eqn:trans_prob}
  \Prb\big(W^i(n+1)=v\mid\F_n\big) = 
  \frac{\exp \Big(\sum_{j=1}^m \alpha_v^{ij} X_v^{j}(n)\Big
    )}{\sum_{u=1}^d \exp\Big(\sum_{j=1}^m \alpha_u^{ij}
    X_u^{j}(n) 
    \Big)},
\end{equation}
where
\begin{equation}
 \label{eqn:strengths}
   \alpha_v^{ij} \in \R
   \quad\text{is such that}\quad
   \alpha_v^{ij} = \alpha_v^{ji}.
\end{equation}
By using \eqref{eqn:occupation_measure}, set
$X^i(n)=\big(X_1^i(n),\ldots,X_d^i(n)\big)$ and then denote by $X$ the
\textit{process of vertex occupation measures} defined as
\begin{equation}\label{theX}
X=\{X(n)\}_{n\ge 0}, \quad\textrm{where}\quad X(n)=(X^1(n),
X^2(n),\ldots, X^m(n)). 
\end{equation}
Notice that $\{X(n)\}_{n \geq 1}$ belongs to the compact convex set
$\mathfrak{D}=\triangle^m=\triangle\times\cdots\times\triangle$, which
equals the $m$-fold Cartesian product of the $(d-1)$-simplex
$\triangle = \{x=(x_v) \in \R^d \mid x_v \geq 0, \sum_v x_v = 1\}$
with itself. In these terms, the process $W$ is completely defined by
specifying the initial condition $X(0)$ and the smooth map
$\pi=(\pi^1_1, \ldots, \pi^1_d, \ldots, \pi^m_1, \ldots,
\pi^m_d):\mathfrak{D}\to\mathfrak{D}$ which at $x=(x^1_1, \ldots,
x^1_d, \ldots, x^m_1, \ldots, x^m_d)$ takes the value
\begin{equation} \label{eqn:pi}
  \pi^i_v(x) = 
   \frac{\exp\big(\sum_{j=1}^m \alpha_v^{ij} x_v^j\big)} {\sum_{u=1}^d
   \exp\big(\sum_{j=1}^m \alpha_u^{ij} x_u^j\big)},
\end{equation}
with $\alpha_v^{ij}$ as the constants satisfying
\eqref{eqn:strengths}. In fact, rewriting \eqref{eqn:trans_prob} in
terms of $\pi$ gives
\[
 \Prb\big(W^i(n+1)=v\mid\F_n\big) =
 \pi_v^i\big(X(n)\big).
\]

It is worth mentioning that (\ref{eqn:trans_prob}) and
(\ref{eqn:strengths}) cover a large variety of interactions which
include different vertex repulsion and attraction strengths between
$m\geq 2$ interacting walks on $G$. In this context, $\alpha_v^{ij}$
stands for the strength of reinforced repulsion (when $\alpha_v^{ij} <
0$) or attraction (when $\alpha_v^{ij} > 0$) between walks $i$ and $j$
at vertex $v$. This may also include self-reinforced interactions
(repulsion or attraction) when $\alpha_v^{ii} \neq 0$.

According to (\ref{eqn:trans_prob}) and (\ref{eqn:strengths}), the
probability of a transition of walk $i$ to a given vertex $v$ at time
$n+1$ depends on the proportions $X_v^{1}(n), X_v^{2}(n), \ldots,
X_v^{m}(n)$ of the visits to vertex $v$ made by all $m$ walks up to
time $n$.  The process $X=\{X(n)\}_{n\ge 0}$ studied throughout
belongs therefore to a family of processes known as vertex reinforced
random walks. There is an extensive literature devoted to
self-attracting reinforced random walks on graphs, see for instance
\cite{P92}, \cite{BT11}, \cite{V01}, and self-repelling walks, see
\cite{T95} and references therein. Several models for interacting
generalised P\'olya urn models have been considered more recently, see
\cite{ACG20}, \cite{CDM16}, \cite{BBCL15}, and \cite{vH16}. Apart from
\cite{C2014} and \cite{CDLM19}, there are relatively few
studies of interacting vertex-reinforced random walks. \cite{C2014}
considers two repelling random walks on finite complete graphs and
focuses on the asymptotic properties of their overlap
measure. \cite{CDLM19} presents a model for several cooperative walks 
on the two vertex graph and describes their synchronisation toward a
common limit.

This article is principally concerned with the asymptotic properties
of the process of vertex occupation measures $X=\{X(n)\}_{n\ge 0}$. A
first step to characterise the long-term behaviour of $X$ consists in
identifying this process with a stochastic approximation.  Stochastic
approximations have been quite effective while dealing with several
self-reinforced processes such as vertex reinforced random walks,
generalised P\'olya urns and population games, see \cite{P07} for a
survey and further references.  In particular, the identification with
a stochastic approximation allows to study the asymptotic behaviour of
the interacting walks by following the dynamical system approach
described in \cite{B96, B99}. Let
$
  \TsD = \big\{x \in \R^{dm}\ \big|\ \sum_v x_v^i = 0\ \text{ for each
  }\  i \in [m]\big\}
$
be the tangent space of $\D$. We show that the process of vertex
occupation measures $X=\{X(n)\}_{n\ge 0}$ can be understood
if we know the asymptotic behaviour of the vector field $F:\D \to
\TsD$ defined by
\begin{equation}
 \label{vectorfieldF}
  F(x)=-x+\pi(x),
\end{equation}
where $\pi$ is given in \eqref{eqn:pi}. To study the long-term
behaviour of the vector field $F$, we adapt the arguments of
\cite{BDFR015_II} and \cite{BDFR015_I} to construct an explicit strict 
Lyapunov function for the vector field.  A key observation for
the construction of this function is based on the fact
that the relative entropy between two solutions of the ordinary
differential equation
\begin{equation}
  \label{eqn:ODE}
  \dot{x}=F(x)
\end{equation}
is strictly decreasing ouside the set of equilibria of the vector
field $F$, namely $F^{-1}(0) = \{x \in \D \mid F(x) = 0\}$.  This
argument was put forward in~\cite{BDFR015_II, BDFR015_I}
while considering non-linear Markov processes with Gibbsian type
interactions. These ideas where also considered and further extended
in \cite{B15}, in order to construct Lyapunov functions for several
examples of self-reinforced processes arising in population games and
vertex reinforced random walks.

The main contribution of this article is to extend the ideas in 
\cite{BDFR015_I, BDFR015_II} to the case of a vector of measures that
represent the vertex occupation defined by many interacting reinforced
random walks.

\section{Statement of the  results}
\subsection{Main results}
Our first result, stated as Theorem~\ref{th:Lyapunov}, shows that the
vector field $F$ defined by \eqref{vectorfieldF} has a strict
Lyapunov function. This result is crucial to understand the long-term
behaviour of the vector field $F$ and the process $X=\{X(n)\}_{n\ge
0}$ defined by (\ref{theX}).

To state our first result, denote by $\Phi=\{\phi_t\}_{t\ge 0}$ the
semi-flow associated with $F$ (see Lemma~\ref{flow}) and by
$F^{-1}(0)\subset\D$ the set of equilibrium points of $F$. We will
need the following definition.
 
\begin{definition}[strict Lyapunov function]\label{slf} A continuous
function $L:{\mathfrak{D}}\to\mathbb{R}$ is a \emph{strict Lyapunov
function} for the vector field $F : \D \to \TsD$ if the function $t\in
[0,\infty)\mapsto L\big(\phi_t(x_0)\big)$ is strictly decreasing for
all $x_0\in {\D}{\setminus}F^{-1}(0)$.
\end{definition}

Now we are able to state our first result.
 
\begin{theorem}\label{th:Lyapunov}
The continuous function $L:\D \to \R$ defined by
\begin{equation}\label{eqn:Lyapunov}
  L(x) =  \sum_{i, v} x^i_v\log( x^i_v)  - \frac{1}{2}\sum_{i,j,v} 
  \alpha_v^{ij} x^i_v x^j_v, \quad \text{with  } \,   0\log(0)
  = 0,
\end{equation}
is a strict Lyapunov function for the vector field
$F$ defined by (\ref{vectorfieldF}).
\end{theorem}

\begin{definition}[linearly stable/unstable equilibria]
Let $\sigma\big(\JacF(x)\big)\subset\mathbb{C}$ be the set of
eigenvalues of the Jacobian matrix of the vector field $F$ at a point
$x \in \D$. We say that an equilibrium point $x$ of $F$ is
\textit{hyperbolic} if $\sigma\big(\JacF(x)\big)$ contains no
eigenvalue with zero real part. An hyperbolic equilibrium point $x$ of
$F$ is \textit{linearly stable} if $\sigma\big(\JacF(x)\big)$ contains
only eigenvalues with negative real parts, otherwise we say that the
hyperbolic equilibrium point $x$ is \textit{linearly unstable}.
\end{definition}

Our second result, stated below, characterises the convergence and
non-convergence of $X$ toward the equilibria of the vector field
$F$. The almost sure convergence described by the last item in this
theorem is a consequence of Theorem~\ref{slf}.

\begin{theorem}\label{thm11} Let $X=\{X(n)\}_{n\geq 0}$ be the process
defined in (\ref{theX}) and $F$ the vector field defined in
(\ref{vectorfieldF}). The following statements hold
\begin{enumerate}[(i), nosep]
\item For each linearly stable equilibrium point $x$ of $F$,
\[
    \mathbb{P}\Big(\lim_{n\to\infty} X(n)=x\Big)>0;
\]
\item  For each linearly unstable equilibrium point $x$ of
$F$,
\[
  \mathbb{P}\Big(\lim_{n\to\infty} X(n)=x\Big)=0;
\]
\item  If the equilibrium points of $F$ are isolated, then 
\[
    \sum_{x\in F^{-1}(0)} \mathbb{P}\Big(\lim_{n\to\infty}
    X(n)=x\Big)=1.
\]
\end{enumerate} 
\end{theorem}

\subsection{Examples}
This section presents a few results concerning several specific
instances of the general model described in the Introduction. The
first two results consider the asymptotic behaviour of any (i.e. $m\geq 
2$) ``weakly interacting'' random walks on any finite complete graph
with $d\geq 2$ vertices. First, by using the global injectivity result
in \cite{DGHN1965}, we show that if the absolute value of the
constants $\alpha^{ij}_v$ are smaller than a certain positive upper
bound, then the vector field $F$ has a unique equilibrium point $x^*$
in $\D$, see Theorem~\ref{thm33}. A direct application of
Theorem~\ref{thm11} in this case shows that the process
$X=\{X(n)\}_{n\ge 0}$ converges almost surely toward $x^*$.  We also
present a sharp result for the case in which the constants
$\{\alpha^{ij}_v:i\neq j\}$ are all equal to each other, see
Theorem~\ref{thm22}.

Apart from these examples, we carry out a complete study for the case
of two repelling random walks on the two-vertex graph. Depending on
the strength of the repulsion, we show that there may be multiple
hyperbolic equilibrium points, see Theorem~\ref{th:TwoRWonK2}. In our
last example, we describe some asymptotic properties of three
repelling random walks defined on $\Z$, in which the repulsion is
determined by the full previous history of the joint process, see
Theorem~\ref{th:3RW}. These processes are defined so that the
probability that a random walk makes a transition in one direction
decreases with the number of times that the other walks made a
transition in that direction.

\subsubsection{Weakly interacting random walks}\label{subsec:weakRW} 
The next two theorems show that if the interaction strength parameters
$\alpha^{ij}_v$ have absolute value less than some positive upper
bound, then the process $X$ converges almost surely to the unique
equilibrium point of the vector field $F$.

Let $\delta_{ij}$ denote Kronecker's delta, i.e., $\delta_{ij}=1$ if
$i=j$ and $\delta_{ij}=0$ if $i\neq j$.

\begin{theorem}\label{thm22} Let
$\alpha_{v}^{ij}=-\beta(1-\delta_{ij})$ for some $\beta>0$ and all
$v \in [d]$ and $i, j \in [m]$. Suppose that at least one
of the following conditions is satisfied:
\begin{align*}
  &d=2, \,\, m\ge 2\quad \textrm{and}\quad \beta \leq 2,&\tag{C1}\\
  & d\ge 2, \,\,m\ge 2 \quad\textrm{and}\quad \beta<\dfrac{4}{d
    (m-1)}.& \tag{C2} 
\end{align*}
Then the process $X = \{X(n)\}_{n\ge 0}$ converges almost surely to
$\big(\frac1d, \frac1d,\ldots,\frac1d\big)$.
\end{theorem}
In Theorem \ref{thm22}, the case $d=2$ can be reduced to an
one-variable problem. The case $d\ge 2, m\ge 2$ is a particular case
of the following general result.

\begin{theorem}\label{thm33} Suppose that $\alpha_v^{ij},$\, $v \in
[d]$ and $i, j \in [m]$, are real numbers satisfying
\eqref{eqn:strengths} and the following
condition holds
\begin{equation}
  \alpha_{v}^{ii}=0 \quad\textrm{and} \quad \sum_{u=1}^d
    \sum_{\substack{j=1\\ j\ne i}}^m  \left| \alpha_u^{ij} \right\vert
    < 4\quad\textrm{for each}\quad (v,i)\in [d]\times[m]
  \tag{C3}
\end{equation}
Then the process $X=\{X(n)\}_{n\ge 0}$ converges almost surely to the
unique equilibrium point of $F$.
\end{theorem}

\subsubsection{An example of planar dynamics: two repelling walks on
  the two-vertex graph}\label{subsec:TwoRMonK2}
The following example considers a relatively simple model consisting
of two exponentially repelling walks, $W^1$ and $W^2$, on the
two-vertex complete graph. These two processes are defined according
to (\ref{eqn:trans_prob}) by setting for all $v
\in [d] = \{1, 2\}$ and $i, j \in [m]=\{1, 2\}$,
\[
  \alpha_v^{ij}
  =
  \begin{cases}
  -\beta, &\text{if }\ i \neq j,\\
  0,      &\text{if }\ i = j,
  \end{cases}
\]
with $\beta \geq 0$. In this case, for $x = (x_1^1$, $x_2^1$, $x_1^2$,
$x_2^2) \in \D$, the coordinate functions of the vector field
$F=(F_1^1$, $F_2^1$, $F_1^2$, $F_2^2)$ defined by (\ref{vectorfieldF})
are explicitly given by
\begin{equation}
  \label{eqn:ODEexplicit}
   F_v^i(x)
   = 
   -x^i_v + \frac{e^{-\beta x^j_v}}{e^{-\beta x^j_1} 
     + e^{-\beta x^j_2}},
  \qquad
  v \in [d], \quad i \in [m], \quad j = 3-i.
\end{equation}

The following theorem provides a complete description for the
asymptotic behaviour for the occupation measure process
$X=\{X(n)\}_{n\geq 0}$, depending on the strength of the repulsion
between the walks $W^1$ and $W^2$.

\begin{theorem}\label{th:TwoRWonK2} If $\beta \in[0, 2)$, then the
vertex occupation measure process $X = \{X(n)\}_{n\ge 0}$ converges
almost surely toward the point $\big(\frac12, \frac12, \frac12,
\frac12\big)$.  If $\beta > 2$, the vertex occupation measure process
$X=\{X(n)\}_{n\ge 0}$ converges almost surely to
\[
  (a, 1-a, 1-a, a) \quad\text{or}\quad (1-a,a,a,1-a),
\]
where $a \in \big(0,\frac12\big)$ is uniquely determined by $\beta$.
\end{theorem}

\subsubsection{Three repelling random walks on $\Z$}\label{subsec:3RW}
Our next example considers the dynamics defined by three repelling
random walks on the two-vertex graph in order to study the asymptotic
behaviour of three random walks defined on $\Z$, reinforced to repell
each other according to the model described in the Introduction.

Let $\{S^i_n$; $i=1,2,3\}_{n\geq 0}$ be the process described by three
random walks on $\Z$ defined as follows. Assume that, for all $i \in
\{1, 2, 3\}$, $S_0^i$ are fixed.  Let $\mathcal A_0$ be the trivial
$\sigma$-algebra and for $n \geq 1$, let $\mathcal A_n = \sigma(\{S_k^1,
S_k^2, S_k^3 : 1 \leq k \leq n\})$ be the natural filtration generated
by these three processes. For $n=0$, the transition probability for
each random walk is set to $\Prb\big(S^i_{1}=S_0^i + 1\mid \mathcal
A_0\big) = \frac{1}{2}$. For $n\geq 1$, the transition probability is
defined as
\begin{equation}\label{eqn:transProb}
\begin{aligned}
\Prb\big(S_{n+1}^i = S_n^i + 1 \,  \big{|}  \,  \mathcal A_n \big)
 &=
   \mu\Big((S_{n}^j - S_{0}^j)/n + (S_{n}^k - S_{0}^k)/n\Big) \\
 &=
1 - \Prb\big(S_{n+1}^i = S_n^i - 1 \, \big{|} \, \mathcal A_n \big),
\end{aligned}
\end{equation}
where $\{i,j,k\} = \{1,2,3\}$, and $\mu: \R \to [0,1]$
is given by the following decreasing function
\begin{equation}\label{eqn:psi}
  \mu(y) = \frac{1}{1 + \exp(\beta y)},
  \qquad \beta \geq 0.
\end{equation}

The following theorem shows that this model has a phase transition at
$\beta=2$. When $\beta < 2$, the three random walks behave
asymptotically as three independent symmetric simple random walks on
$\Z$. When $\beta > 2$, there are always two random walks such that
one of them diverges to $-\infty$ while the other one diverges to
$+\infty$. The third walk may behave asymptotically as a simple
symmetric walk.

\begin{theorem}\label{th:3RW} If $\beta < 2$, for any $i \in \{1, 2,
  3\}$, then, with probability one, it holds that 
\begin{equation}\label{eq:betame2}
  \lim_{n\to\infty} \Prb\Big(S_{n+1}^i - S_n^i = 1 \,\Big|\,  \mathcal
  A_n\Big) = \frac{1}{2}.
\end{equation}

For sufficiently large $\beta$, with positive probability, there are
two random walks $i$, $j$ such that
\begin{equation}\label{eq:betama2}
  \lim_{n\to\infty} S^i_n = -\lim_{n\to\infty} S^j_n = \infty,
\end{equation}
and for $S_n^k$, $k \notin \{i,j\}$, it holds that $S_n^k$ behaves
assymptotically as a simple symmetric random walk, that is,
(\ref{eq:betame2}) holds for $i = k$.
\end{theorem}

\section{Proof of Theorem~\ref{th:Lyapunov}: The Lyapunov function}
This section presents the proof of Theorem
\ref{th:Lyapunov}. Throughout let $\pi:\D\to \D$ be the smooth map
defined in \eqref{eqn:pi}, $F:\mathfrak{D}\to\TsD$ the smooth vector
field defined in (\ref{vectorfieldF}) and $L:\mathfrak{D}\to\mathbb{R}$
the continuous function defined in \eqref{eqn:Lyapunov}.  The
\textit{boundary} and the \textit{interior} of $\mathfrak{D}$ are
respectively, the sets
\[
  \partial\D = \Big\{\, x \in \D\ \Big|\ \prod_{i,v} x_v^i =
  0\, \Big\},
  \qquad
  \iD = \D{\setminus}\partial\D.
\]

\begin{lemma}\label{flow} There exists a uniquely defined
one-parameter family $\Phi=\{\phi_t\}_{t\ge 0}$ of self-maps of
$\mathfrak{D}$, called the semi-flow associated with  $F$, such that
the map $(t,x)\mapsto \phi_t(x)$ is smooth, and the following holds
for each $x_0\in\D:$ 
\begin{enumerate}[(i), nosep]
\item $\phi_0 (x_0 )=x_0$ and $\phi_t (x_0 )\in\iD$
  for all $t>0$, 
\item $\frac{d}{dt} \phi_t (x_0 )=F\big(\phi_t (x_0 )\big)$
  for all $t\ge 0$. 
\end{enumerate}
\end{lemma}
\begin{proof}
Let $x_0\in\partial\mathfrak{D}$. By \eqref{eqn:pi},
$\pi\big(x_0\big)\in \iD$. By (\ref{vectorfieldF}), $F(x_0)$ is the
displacement vector from $x_0\in\partial{\mathfrak D}$ to
$\pi\big(x_0\big)\in\iD$. Hence, by the convexity of $\mathfrak{D}$,
we have that $F(x_0)$ points towards the interior of $\mathfrak{D}$,
i.e, $\mathfrak{D}$ is invariant by $F$.
Moreover, since $F$ is smooth, we have that $F$ is locally
Lipschitz. A widely known result of the theory of ordinary
differential equations (see \citep[Theorem 3.3]{Kha92}) now asserts
that every locally Lipschitz vector field defined on an invariant
compact set admits a uniquely defined semi-flow.
\end{proof}

For each $x\in \D$, let $\Gamma(x)$ be the $md\times md$ matrix  
\[
  \Gamma(x) = -I + \Pi(x),
\]
where $I$ denotes the $md\times md$ identity matrix and $\Pi(x)$ is
defined as 
\begin{equation}\label{trans}
 \Pi(x)
 =
 \begin{bmatrix}
    \Pi^1(x)   & \mathbf{0}  & \cdots   & \mathbf{0}    \\[.27em]
    \mathbf{0} & \Pi^2(x)    & \cdots   & \mathbf{0}    \\[.27em]
    \vdots     & \vdots      & \ddots   & \vdots        \\[.27em]
    \mathbf{0} & \mathbf{0}  & \cdots   & \Pi^m(x)
  \end{bmatrix},
\end{equation}
where $\mathbf 0$ is the $d\times d$ zero matrix and for each $i\in
[m]$, the $d\times d$ block matrices $\Pi^i(x)$ are given by
\begin{equation}
  \label{eqn:Pi_supi}
  \Pi^i(x)
  =
  \begin{bmatrix}
    \pi_1^i(x)  & \pi_2^i(x) & \cdots & \pi_d^i(x)\\[0.2em] 
    \pi_1^i(x)  & \pi_2^i(x) & \cdots & \pi_d^i(x)\\
    \vdots      & \vdots     &        & \vdots\\
    \pi_1^i(x)  & \pi_2^i(x) & \cdots & \pi_d^i(x)
 \end{bmatrix}.
\end{equation}

In the next lemma, $\pi(x)$ denotes the vector
$\pi(x)=\big(\pi_1^1(x)$, $\ldots$, $\pi_d^1(x)$, $\ldots$,
$\pi_1^m(x),\ldots$, $\pi_d^m(x)\big)$.
  
\begin{lemma}\label{la1} $\pi(x)\Gamma(x)=\mathbf{0}$ for all $x\in
 \D$. 
\end{lemma}
\begin{proof} The $v$,$i$-entry of $\pi(x)\Gamma(x)$ is
\[
  \pi_1^i \pi_v^i +\ldots+\pi_{v-1}^i
  \pi_v^i+\pi_v^i(\pi_v^i-1)+\pi_{v+1}^i\pi_v^i+\ldots+\pi_d^i
  \pi_v^i
  =
  \Big(\sum_{v}\pi_v^i\Big)\pi_v^i-\pi_v^i=\pi_v^i-\pi_v^i=0,
\]
where we have omitted  $x$ in $\pi_u^i(x)$ to save space.
\end{proof}

In the next lemma, $x(t)$ denotes the vector
$x(t) = \big(x_1^1(t), \ldots, x_d^1(t), \ldots , x_1^m(t), \ldots,
x_d^m(t)\big)$.

\begin{lemma}\label{la2} Given $x_0\in\D$, let $x(t)=\phi_t(x_0)$ for
all $t\ge 0$. Then
\begin{equation*}
  \frac{d}{dt}x(t) = x(t) \Gamma(x(t))\quad \textrm{for all}\quad t\ge
  0.
\end{equation*}
\end{lemma}
\begin{proof} By the item $(\mathit{ii})$ of Lemma \ref{flow}, we have
that the $v$,$i$-entry of $\frac{d}{dt} x(t)$ is  
\begin{equation}\label{123a}
  \frac{d}{dt}x_v^i(t)
  =
  F_v^i\big(x(t)\big)=-x_v^i(t)+\pi_i^v\big(x(t)\big)\cdot 
\end{equation}
On the other hand, by Lemma \ref{la1}, we have the $v$,$i$-entry of
the vector
\[
  x(t)\Gamma\big(x(t)\big)
  =
  \Big(x(t)-\pi\big(x(t)\big)\Big)\Gamma\big(x(t)\big)
\]
is given by (the expressions $(t)$ and $x(t)$ were omitted to save
space)
\begin{align*}\label{123b}
  (x_1^i-\pi_1^i)\pi_v^i &+ \ldots + (x_{v-1}^i -
    \pi_{v-1}^i)\pi_v^i                           
  + (x_v^i - \pi_v^i)                             
    (\pi_v^i-1)                                   \\
  &+ (x_{v+1}^i-\pi_{v+1}^i)
    \pi_v^i+\ldots+(x_d^i-\pi_d^i) \pi_v^i        \\
  &=(x_v^i-\pi_v^i)(-1)+ \sum_{u}
    (x_u^i-\pi_u^i)\pi_v^i                        \\
  &=-x_v^i+\pi_v^i
    +\pi_v^i\sum_{u}F_u^i(x)=-x_v^i+\pi_v^i,
\end{align*}
where $\sum_{u}F_u^i(x)=0$ because $F(x)\in T\D$.

Putting it all together, we have that
\[
   \frac{d}{dt}x_{v}^i(t)
  =
  -x_v^i(t) + \pi_v^i\big(x(t)\big)
  =
  \left[x(t)\Gamma\big(x(t)\big)\right]_v^i,
\]
which concludes the proof.
\end{proof}

\begin{definition}\label{def:RelativeEntropy} Given two vectors of
probability measures $x, y \in \iD$, we write $x=(x^1,
\ldots, x^m)$ and $y = (y^1, \ldots, y^m)$, where $x^i = (x_1^i,
\ldots, x_d^i)$, $y^i = (y_1^i, \ldots, y_d^i)$, $i \in [m]$. The
relative entropy of $x$ with respect to $y$ is defined as
\[
  \Ent\Big(\frac{x}{y}\Big)
  =
  \sum_{i=1}^m \Ent\Big(\frac{x^i}{y^i}\Big),
  \quad \text{where}\quad 
  \Ent\Big(\frac{x^i}{y^i}\Big) = 
  \sum_{v} x_v^i \log\bigg(\frac{x_v^i}{y_v^i}\bigg)
\]
and $\log$ is the natural logarithm. 
\end{definition}

The following lemmas will be used in the proof of
Theorem~\ref{th:Lyapunov}.

\begin{lemma}\label{lem:Step1} Given $x_0\in\D$, let
$x(t)=\phi_t\big(x_0\big)$ for all $t> 0$. Then
\[
  \frac{d}{dt} L(x(t))\bigg|_{t=\tau}
  =
  \frac{d}{dt}
  \Ent\bigg(\frac{x(t)}{\pi\big(x(\tau)\big)}\bigg)\bigg|_{t=\tau}
  \quad \textrm{for all}\quad \tau>0.\]
\end{lemma}
\begin{proof} 
By Lemma \ref{flow}, it follows that $x(t)\in \iD$ for all
$t>0$, hence $\textrm{Ent}\big(x(t)/\pi(x(\tau))\big)$ is well-defined
for all $t,\tau>0$. Computing the spatial derivatives of $L$ at
$x=(x_v^i)\in\iD$ and using \eqref{eqn:strengths} lead to
\[
  \frac{\partial L}{\partial x_v^i}(x) = \log(x_v^i) + 1 - \sum_j
  \alpha_v^{ij} x_v^j.
\]
Hereafter, let $w=x(\tau)$ for some arbitrary but fixed
$\tau>0$. Computing the derivative of $t\mapsto L\big(x(t)\big)$ at
$\tau$ yields
\begin{align}
  \frac{d}{dt} L\big(x(t)\big)\bigg|_{t=\tau}
  &= 
  \sum_{i, v} \frac{\partial L}{\partial
    x_v^i}(w) \frac{d}{dt} x_v^i(t)\bigg|_{t=\tau}  \notag\\
  &=
  \sum_{i, v} \bigg(\log(w_v^i) + 1 - \sum_j \alpha_v^{ij} w_v^j
    \bigg)\frac{d}{dt} x_v^i(t)\bigg|_{t=\tau} \notag\\
  &=
  \sum_{i, v} \bigg(\log(w_v^i) - \sum_j \alpha_v^{ij} w_v^j \bigg)
    \frac{d}{dt} x_v^i(t)\bigg|_{t=\tau} \notag\\
    \label{eqn:Deriv_Lyapunov}
  &=
    \sum_{i,v}\log(w_v^i) \frac{d}{dt} x_v^i(t)\bigg|_{t=\tau} - 
    \sum_{i,j,v} \alpha_v^{ij} w_v^j \frac{d}{dt}
    x_v^i(t)\bigg|_{t=\tau},
\end{align}
where the third equality above holds because, since $F\big(w\big)\in
T\D$, we have that
\begin{equation}
  \label{eqn:sumField}
  \sum_{i,v} \frac{d}{dt} x_v^i(t)\bigg|_{t=\tau} = \sum_{i,v}
  F_v^i(w) = 0. 
\end{equation}

On the other hand, the derivative of the entropy between $x(t)$ and
$\pi(w)$ gives
\begin{equation}
 \label{eqn:Deriv_Entropy}
   \frac{d}{dt}
   \Ent\bigg(\frac{x(t)}{\pi(x(\tau))}\bigg) \bigg|_{t=\tau}
   = 
  \frac{d}{dt}\sum_{i, v} x_v^i(t)\log x_v^i(t)\bigg|_{t=\tau}
 -\frac{d}{dt}\sum_{i, v} x_v^i(t)\log\pi^i_v(w)\bigg|_{t=\tau}. 
\end{equation}

Using (\ref{eqn:sumField}), the first term at the right-hand side of
(\ref{eqn:Deriv_Entropy}) equals
\begin{equation} \label{eqn:Deriv_Lyapunov-1oterm}
 \frac{d}{dt}\sum_{i, v} x_v^i(t)\log x_v^i(t)\bigg|_{t=\tau}
 =
 \sum_{i,v} \log(w_v^i)\frac{d}{dt} x_v^i(t)\bigg|_{t=\tau}.
\end{equation}

To analyse the second term at the right-hand side of
(\ref{eqn:Deriv_Entropy}), note that $\pi^i_v(w)$ can be written as
\begin{equation}\label{20201}
   \pi^i_v(w) = e^{\sum_j \alpha^{ij}_v w_v^j}\big/Z_i(w),
\end{equation}
where $Z_i(w)$ is the normalising factor, that is, 
\[
  Z_i(w) = \sum_u e^{\sum_j \alpha^{ij}_u w_u^j}.
\]
By \eqref{20201}, the second term at the right-hand side of
(\ref{eqn:Deriv_Entropy}) becomes
\begin{align}
  -\frac{d}{dt}\sum_{i, v} x_v^i(t)\log\pi^i_v(w)\bigg|_{t=\tau}
  &=
  -\frac{d}{dt}\sum_{i,v}  x_v^i(t)
    \log\bigg(\frac{e^{\sum_j
       \alpha_v^{ij}w_v^j}}{Z_i(w)}\bigg)\bigg|_{t=\tau}
    \notag\\ 
  &=
  -\sum_{i,j,v} \alpha_v^{ij}w_v^j \frac{d}{dt} x_v^i(t)\bigg|_{t=\tau}
  +\sum_{i,v}\log(Z_i(w))\frac{d}{dt}x_v^i(t)\bigg|_{t=\tau} \notag\\
   \label{eqn:Deriv_Lyapunov-2oterm}
  &=
  -\sum_{i,j,v}\alpha_v^{ij}w_v^j \frac{d}{dt}
    x_v^i(t)\bigg|_{t=\tau}, 
\end{align}
where the last equality holds because $\frac{d}{dt}
x_v^i(t)\big|_{t=\tau}=F_v^i(w)$ and $F\big(w\big)\in T\D$.

Comparing (\ref{eqn:Deriv_Lyapunov}) with (\ref{eqn:Deriv_Entropy}),
(\ref{eqn:Deriv_Lyapunov-1oterm}), and
(\ref{eqn:Deriv_Lyapunov-2oterm}) concludes the proof.
\end{proof}

The following lemma is an adapted version of Lemma 3.1 in
\cite{BDFR015_I}. 

\begin{lemma}\label{lem:Step2}
Given $x_0\in\D$, let $x(t)=\phi_t\big(x_0\big)$ for all $t> 0$. Let
$x^i(t) = (x_1^i(t), \ldots, x_d^i(t))$ and $\pi^i(x(t)) =
(\pi_1^i(x(t)), \ldots, \pi_d^i(x(t)))$. The following inequality
holds
\begin{equation}
  \label{eqn:Lyap_via_Ent2}
  \frac{d}{dt}
  \Ent\bigg(\frac{x^i(t)}{\pi^i(x(\tau))}\bigg)\bigg|_{t=\tau} \leq 0, 
  \quad\forall
  \tau>0,\quad\forall i\in [m]. 
\end{equation}
Moreover,  
\begin{equation}\label{eqn:Lyap_via_Ent_iff}
  \exists \tau>0,\,\, \forall i\in [m],\,\,
  \frac{d}{dt}
  \Ent\bigg(\frac{x^i(t)}{\pi^i(x(\tau))}\bigg)\bigg|_{t=\tau} = 0 
   \quad
    \text{if and only if} \quad x_0\in F^{-1}(0).
\end{equation}
\end{lemma}
\begin{proof}
By Lemma \ref{flow}, $x(t)\in\iD$ for all $t>0$, thus 
$\textrm{Ent}\,\left(x^i(t)/\pi^i(x(\tau))\right)$ is
well-defined for all $t,\tau>0$. Hereafter, let $w = x(\tau) \in \iD$
for some arbitrary but fixed $\tau>0$. By Lemma \ref{la2} and
\eqref{trans},
\begin{equation}\label{kolmogorov-eq-i}
  \frac{d}{dt} x^i(t)\bigg|_{t=\tau} = x^i(\tau)\Gamma^i(w),
  \quad i\in [m], 
\end{equation}
where 
\begin{equation}\label{gi} \Gamma^i(w)=-I_d + \Pi^i(w), 
\end{equation}
$I_d$ is the $d\times d$ identity matrix and $\Pi^i(w)$ is the
matrix in (\ref{eqn:Pi_supi})  with $x$ replaced by $w$.

Let $\ell : (0,\infty) \to [0, \infty)$ be the continuous function
defined as $\ell(z) = z\log z - z + 1$. The inequality
(\ref{eqn:Lyap_via_Ent2}) will be shown by assuming, and proving
later, that
\begin{equation}
 \label{eqn:dEntequlell}
   \frac{d}{d t}\Ent\bigg(\frac{x^i(t)}{\pi^i(w)}\bigg)\bigg|_{t=\tau} 
  = 
  -\sum_{u, v: u\neq v} \ell\bigg(
  \frac{x_v^i(\tau)\pi_u^i(w)}{x_u^i(\tau)\pi_v^i(w)}
  \bigg)
  x_u^i(\tau) \frac{\pi_v^i(w)}{\pi_u^i(w)}\Gamma^i_{vu}(w),
\end{equation}
where $\Gamma^i_{vu}(w)$ is the $v$,$u$-entry of the matrix
$\Gamma^i(w)$ defined in \eqref{gi}.

Assuming (\ref{eqn:dEntequlell}), inequality (\ref{eqn:Lyap_via_Ent2})
is an immediate consequence of the following three facts:
\begin{enumerate}[$1.$]
\item $\pi_v^i(w) > 0$ for all $v \in [d]$,
\item $x_v^i(\tau) > 0$ for all $v\in [d]$, and
\item $\ell((0,\infty)) \subset [0,\infty)$.
\end{enumerate}
The first assertion follows from the definition of $\pi_v^i$. The
second holds because, $x(t)=\phi_t\big(x_0\big)\in\iD$ for
all $t>0$. The third is trivial.  Next we verify
(\ref{eqn:Lyap_via_Ent_iff}). If $x_0\in F^{-1}(0)$, then $x(t)=x_0$
and $F\big(x(t)\big)=0$ for all $t\ge 0$. In particular, $x^i(t) =
\pi^i(w)$, for all $t>0$, then $\frac{d}{dt} \Ent\big(x^i(t)/\pi^i(w)
\big)\big|_{t=\tau} = 0$, for all $t>0$, because the argument of the
function $\ell$ in (\ref{eqn:dEntequlell}) equals $1$ for all $u$,
$v$, and $\ell(1) = 0$. Converselly, assume that
$\frac{d}{dt}\Ent\big(x^i(t)/\pi^i(w) \big)\big|_{t=\tau} = 0$. By the
facts $a$) and $b$) above and also because $\Gamma^i_{vu}(w)>0$ for all
$u\neq v$, we have that $\sum_{u,v:u\neq v}\ell(z_{vu}^i)=0$, where
\[
  z_{vu}^i =\frac{x_v^i(\tau)\pi_u^i(w)}{x_u^i(\tau)\pi_v^i(w)}. 
\]
Now the fact c) implies that $\ell(z_{vu}^i)=0$ for all $v\neq
u$. This implies that $z_{vu}^i = 1$ for all for all $v, u \in V$. As
a consequence, $x_v^i(\tau)\pi_u^i(w) = x_u^i(\tau)\pi_v^i(w)$ for all
vertices $v$ and $u$. Summing both sides of the previous equality over
all $u \in [d]$ yields $x_v^i(\tau) = \pi_v^i(w)$ for all $v$, that is,
$x^i(\tau) = \pi^i(w)=\pi^i(x(\tau))$, for all $i\in [m]$. This
implies that  $x(\tau)=\pi(x(\tau))$, and hence that $F(x(\tau))=0$,
i.e., $x(\tau)$ is an equilibrium point of $F$. Hence,
$x(t)=x(\tau)=x_0$ for all $t>0$.

It remains to show that  (\ref{eqn:dEntequlell}) is true. First note
that since $\ell(z^i_{vv})=0$, we can replace $\sum_{u,v:u\neq v}$ by
$\sum_{u,v}$ in the right-hand side of
(\ref{eqn:dEntequlell}). Applying the definition of $\ell$ and
rearranging terms, the right-hand side of (\ref{eqn:dEntequlell})
equals 
\begin{equation}
\begin{aligned}
  \sum_{v, u \in V} \Bigg (x_v^i(\tau) &+ x_v^i(\tau) \log \Bigg (
  \frac{x_u^i(\tau)}{\pi_u^i(w)} \Bigg ) -   x_u^i(\tau)
  \frac{\pi_v^i(w)}{\pi_u^i(w)} \Bigg )  \Gamma^i_{vu}(w)   \label{dl1}  \\ 
  &
   - \sum_{v, u \in V}   x_v^i(\tau) \log \Bigg (
   \frac{x_v^i(\tau)}{\pi_v^i(w)} \Bigg )  \Gamma^i_{vu}(w). 
\end{aligned}
\end{equation}

Since for each $v\in V$, $\sum_{u \in V}\Gamma_{vu}^{i}(w) = 0$, the
second line of (\ref{dl1}) equals zero.  In addition, by Lemma
\ref{la1}, it follows that $\sum_{v \in V} \pi_v^i(w)\Gamma^i_{vu}(w)
= 0$ for each $u\in V$. Taking these two facts into account together
with (\ref{kolmogorov-eq-i}) shows that (\ref{dl1}) reduces to
\begin{align*}
  \sum_{u, v \in V} \Bigg[x_v^i(\tau) 
  \ +\ & 
  x_v^i(\tau) \log \Bigg (\frac{x_u^i(\tau)}{\pi_u^i(w)} \Bigg )\Bigg]
         \Gamma^i_{vu}(w)   \\ 
  &=
  \sum_{u \in V} \Bigg [ \frac{d}{d t} x_u^i(t)\bigg|_{t=\tau} +  \log 
    \Bigg( \frac{x_u^i(t)}{\pi_u^i(w)} \Bigg ) \frac{d}{d
    t} x_u^i(t)\bigg|_{t=\tau}\Bigg] \\ 
  &=
  \frac{d}{d t} \Bigg [1+ \sum_{u \in V}  x_u^i(t) \log 
    \Bigg ( \frac{x_u^i(t)}{\pi_u^i(w)} \Bigg )  \Bigg]\bigg|_{t=\tau}   
  =   
 \frac{d}{dt}\Ent\bigg(\frac{x^i(t)}{\pi^i(w)}\bigg)\bigg|_{t=\tau}.
 \qedhere
\end{align*}
\end{proof}

\begin{proof}[Proof of Theorem \ref{th:Lyapunov}] To prove that
$L$ is a strict Lyapunov function, we will show that $t \in [0,
\infty) \mapsto L\big(\phi_t(x_0)\big)$ is strictly decreasing. See
Definition \ref{slf}. First we show how to combine Definition
\ref{def:RelativeEntropy} with Lemmas~\ref{lem:Step1} and
\ref{lem:Step2} to prove the following claim:
\begin{equation}\label{101}
\frac{d}{dt} L\big(\phi_t(x_0)\big)\bigg|_{t=\tau} < 0,\quad\forall
x_0\in\D{\setminus} F^{-1}(0),\quad\forall \tau\in (0,\infty).
\end{equation}
In fact, let $x_0\in \D{\setminus}F^{-1}(0)$. By
Definition~\ref{def:RelativeEntropy} and Lemma~\ref{lem:Step1} we have
that
\[
  \frac{d}{dt} L\big(\phi_t(x_0)\big)\bigg|_{t=\tau}
  =
  \sum_{i=1}^m \frac{d}{dt}
  \Ent\bigg(\frac{\phi^i_t(x_0)}{
    \pi^i\big(\phi_\tau(x_0)\big)}\bigg)\bigg|_{t=\tau}
  \quad \text{for all }\ \tau > 0.
\]
Hence, using \eqref{eqn:Lyap_via_Ent_iff} gives
\[
  \frac{d}{dt} L\big(\phi_t(x_0)\big)\bigg|_{t=\tau} \neq 0, 
  \qquad \forall x_0 \in \D{\setminus}\Eq,
  \quad \forall \tau \in (0,\infty),
\]
and by \eqref{eqn:Lyap_via_Ent2} we have
\[
  \frac{d}{dt} L\big(\phi_t(x_0)\big)\bigg|_{t=\tau} \le 0
  \quad \forall \tau \in (0,\infty).
\]
These two assertions combined prove the claim.
  
Let $x_0\in \D{\setminus} F^{-1}(0)$. By (\ref{101}) and by the
continuity of $t\mapsto L\big(\phi_t(x_0)\big)$ at $0$, we have that
the function $t\in [0,\infty)\mapsto L\big(\phi_t(x_0)\big)$ is
strictly decreasing, showing that $L$ is a Lyapunov function for the
vector field $F$.
\end{proof}

\section{Proof of Theorem~\ref{thm11}: Stochastic approximations}
In this section, we show how the asymptotic behaviour of the process
of empirical vertex occupation measures $X=\{X(n)\}_{n\ge 0}$ defined
in (\ref{theX}) is related to the asymptotic behaviour of the ordinary
differential equation (\ref{eqn:ODE}) where $F$ is the vector field
defined in \eqref{vectorfieldF}. A formulation based on dynamical
systems theory that makes precise the connection between the process
$X=\{X(n)\}_{n\ge 0}$ and the semi-flow $\Phi = \{\phi_t\}_{t\ge 0}$ 
induced by the vector field $F$ has been developed in \cite{B96},
\cite{B99}. This connection will be established by
Lemma~\ref{lm:LimitSet} stated bellow.

For each $n\geq 1$ define
\begin{equation}
  \label{eqn:xi}
  \xi(n) = (\xi_v^i(n); 1 \leq i \leq m, 1 \leq v \leq d)
  \quad\text{where}\quad
  \xi_v^i(n) = \Ind\{W^i(n+1) = v\}.
\end{equation}

The following lemma allows to identify $X=\{X(n)\}_{n\ge 0}$ with a
specific process known as a stochastic approximation. This step is key
to the general approach followed throughout.

\begin{lemma}\label{lem:XSA}
The process $X=\{X(n)\}_{n\ge 0}$ satisfies the recursion
\begin{equation}
  \label{eqn:SA} 
  X(n+1) - X(n) = \gamma_n(F(X(n))+U_n),
\end{equation}
where
\begin{equation}
	\label{eqn:gamma_and_U}
  \gamma_n = \frac{1}{n+d+1},
  \qquad
  U_n = \xi(n) - \E[\xi(n)\mid\F_n],
\end{equation}
and $F$ is the vector field defined in \eqref{vectorfieldF}.
\end{lemma}

\begin{proof}
The increment of the occupation measure for the vertex $v \in V$ at
time $n+1$ by the $i$-th random walk is given by
\begin{align*}
  X^i_v(n+1) - X^i_v(n) 
  &= 
    \frac{1 + \sum_{k=0}^{n-1} \xi^i_v(k) + \xi^i_v(n)}{d+n+1} - 
    \frac{1 + \sum_{k=0}^{n-1} \xi^i_v(k)}{d+n} \\
  &= 
    \frac{1}{d+n+1} \bigg(-\frac{1 + \sum_{k=0}^{n-1}
    \xi^i_v(k)}{d+n} +  \xi^i_v(n) \bigg) \\
  &=
    \frac{1}{d+n+1}\Big(-X^i_v(n) + \xi^i_v(n)\Big).
\end{align*}
Setting $\gamma_n = (n+d+1)^{-1}$  and	using $\xi$ as defined in
(\ref{eqn:xi}) leads to
\begin{align*}
  X(n+1) - X(n)
  &=
    \gamma_n(-X(n) + \xi(n)) \\
  &=
    \gamma_n\Big\{\Big(-X(n)+ \E[\xi(n)\mid\F_n]\Big) +
    \Big(\xi(n)-\E[\xi(n)\mid\F_n]\Big)\Big\} \\
  &=
    \gamma_n\Big\{\Big(-X(n)+ \E[\xi(n)\mid\F_n]\Big) +
    U_n \Big\}.
\end{align*}
Now, according to (\ref{eqn:pi}),
\[
  \E[\xi(n)\mid\F_n]
  =
  \big(
  \Prb(W^i_{n+1} = v\mid\F_n); 1\leq i \leq m, 1\leq v\leq d 
  \big)
  =
  \pi(X(n)).
\]
Substituting this into the expression for the increments of $X$ gives
$X(n+1) - X(n) = \gamma_n\big\{\big(-X(n)+ \pi(X(n))\big) + U_n
\big\}$, which, by using the definition of $F$ in (\ref{vectorfieldF}), 
concludes the proof.
\end{proof}

The following two definitions are necessary to state
Lemma~\ref{lm:LimitSet}.

\begin{definition}[Chain-recurrent set]\label{def:ChainRec}
Let $\delta >0$, $T>0$. A $(\delta, T)$-pseudo orbit from $x \in \D$
to $y \in \D$ is a finite sequence of partial orbits $\{\phi_t(y_i) :
0 \leq t \leq t_i\}$; $i=0, \ldots, k-1$; $t_i\geq T$ of the semi-flow
$\Phi= \{\phi_t\}_{t\geq 0}$ such that
\[
  \|y_0 - x\| < \delta,
  \qquad
  \|\phi_{t_i}(y_i) -  y_{i+1}\| < \delta,\ \ i=0, \ldots, k-1,
  \quad\text{and}\quad
  y_k = y.
\]
A point $x \in \D$ is \emph{chain-recurrent} if for every $\delta>0$
and $T>0$ there is a $(\delta, T)$-pseudo orbit from $x$ to
itself. The set of chain-recurrent points of $\Phi$ is denoted by
$\CR$.
\end{definition}

It follows that $\CR$ is closed, positively invariant and such that
$\Eq \subset \CR$.

\begin{definition}\label{def:Limit_Set}
Let $\sL\big(\{X(n)\}\big)$ be the limit set of the
stochastic approximation process $\{ X(n)\}_{n\geq
0}$. That is, for any point $\omega \in
\Omega$, the value of $\sL\big(\{X(n)\}\big)$ at $\omega$ is given by the 
set of poins $x \in \R^{md}$ for which $\lim_{k\to
\infty} X(n_k, \omega) = x$, for some strictly increasing sequence of
integers $\{n_k\}_{k \in \N}$.
\end{definition}

\begin{lemma}\label{lm:LimitSet} 
Let $X = \{X(n)\}_{n\ge 0}$ be the occupation measure process
satisfying the recursion in \eqref{eqn:SA}. The following hold
\begin{enumerate}[(i), nosep]
\item $\{X(n)\}_{n\geq 0}$ is bounded, 
\item $\lim_{n\to\infty}\gamma_n=0$, $\sum_{n\geq 0}
    \gamma_n = \infty$, and $\sum_{n\geq 0} \gamma_n^2 <
    \infty$, 
\item\label{as:KushnerLemma} For each $T>0$, almost surely it holds
that
\[
  \lim_{n\to\infty}
  \Bigg(\sup_{\{\, r\, :\, 0\, \leq\, \tau_r - \tau_n\, \leq\, T\, \}}
  \Bigg\|\sum_{k=n}^{r-1} \gamma_k U_k\Bigg\|
  \Bigg) = 0,
\]
where $\tau_0=0$ and $\tau_n = \sum_{k=0}^{n-1} \gamma_k$.
\item The set $\mathfrak L\big(\{X(n)\}\big)$ is almost 
surely connected and included in $\CR$, the chain-recurrent set of the
semi-flow induced by the vector field $F$ in (\ref{vectorfieldF}). 
\end{enumerate}
\end{lemma}
\begin{proof}
The proof of $(\mathit{iv})$ follows from Theorem 1.2 in \cite{B96}
together with Corollary 1.2 to Theorem 1.1 in \cite{BH95}, and relies
on the items $(\mathit{i})$-$(\mathit{iii})$. The assertions in 
$(\mathit{i})$ and $(\mathit{ii})$ are immediate. We will prove
$(\mathit{iii})$. Let $M_n = \sum_{k=0}^n \gamma_k U_k$. The process
$\{M_n\}_{n\geq 0}$ is a martingale with respect to $\{\F_n, n\geq
0\}$, that is
\[  
 \E[M_{n+1}\mid\F_{n+1}] = \sum_{k=0}^n \gamma_k U_k +
   \gamma_{n+1}\E[U_{n+1}\mid\F_{n+1}] = M_n.
\]
Observe that
\begin{align*}      
  \E\big[\|M_{n+1} - M_n\|^2\big|\ \F_{n+1}\big] &= \gamma_{n+1}^2
\E\big[\|U_{n+1}\|^2\big|\ \F_{n+1}\big] \\ &\leq \gamma_{n+1}^2
\Big(\sum_{i,v} \xi^i_v(n+1)\Big)^2 \\ &\leq (md)^2 \gamma^2_{n+1}.
\end{align*}
Using Doob's decomposition for the sub-martingale $M_n^2$, let $\{A_n,
n\geq 1\}$ be a predictable increasing sequence defined by $A_{n+1} =
M_n^2 + M_n$ with $A_1=0$. The conditional variance formula for the
increment $M_{n+1} - M_n$ gives
\[  
  A_{n+2}-A_{n+1}
  =
  \E\big[M_{n+1}^2\big|\ \F_n\big] - M_n^2 =
  \E[\|M_{n+1}-M_n\|^2\mid  \F_{n+1}],
\]
and hence for any $n$,
\[
   A_{n+2}
  =
   \sum_{k=0}^n \E\big[\|M_{k+1}
   - M_k\|^2\big|\ \F_{n+1}\big]
  \leq
    (md)^2 \sum_{k=0}^n \gamma_{k+1}^2.
\]
Passing to the limit $n\to\infty$ shows that almost surely $A_\infty <
\infty$. According to Theorem~5.4.9 in \cite{D2010}, this in turn
implies that $M_n$ converges almost surely to a finite limit and hence
that $\{M_n\}_{n\geq 0}$ is a Cauchy sequence. This is sufficient to
conclude the proof.
\end{proof}

The proof of the second item of Theorem~\ref{thm11}, concerning the
non-convergence toward linearly unstable equilibria, makes use of
the following lemma. For $w \in \R$, let $w^+ = \max\{w,0\}$
and $w^- = \max\{-w, 0\}$.

\begin{lemma}\label{lem:boundfrombellow}
Let $x^*$ be a linearly unstable equilibrium of the vector field $F$
defined by (\ref{vectorfieldF}). There is a
neighborhood  $\Bhood(x^*)$ of $x^*$ and a constant $c > 0$  
\begin{equation}
  \label{eqn:NoiseboundBellow} 
  \E\Big[\big\langle \theta, U_n \big\rangle^+ \, \Big| \, X(n) = x,
  \, \F_n \Big] 
  \geq c 
\end{equation}
for every $n > 0$, every $x \in \Bhood(x^*)$, and every $\theta \in
\TsD_1$.
\end{lemma}

\begin{proof}
It is sufficient to show that, for all $n > 0$, $x \in \D$, and
$\theta \in \TsD_1$, we have that
\begin{equation}\label{Edesig}
 \E\Big[ \big\langle\theta, U_n  
   \big\rangle^+ \Big{|} X(n) = x,  \,\F_n \Big] \geq s(x),
\end{equation}
where $s:  \D \to \mathbb{R}$ is a continuous function such that
$s(x^*) > 0$.

Let 
\begin{equation}\label{eqn:defs}
 s(x) = \frac{1}{2 m d} \Big(\min_{i,v} \pi_{v}^{i} \big(x\big)
 \Big)^{m+1}.
\end{equation}

Clearly $s$, as defined in (\ref{eqn:defs}), is continuous. Since
$F(x) = -x + \pi(x)$ and since $F(x^*) = 0$, we have that $\pi(x^*) =
x^*$. As a consequence, $s(x_*) > 0$, where the previous inequality
holds because $x^*$ belongs to the interior of $\D$.

It remains show (\ref{Edesig}).  Let $\theta \in \TsD_1$. For each
walk $i \in [m]$, choose a vertex $v^i \in \{1,2,\ldots, d\}$, such
that
\[
  \theta_{v^i}^i  =  \max_v \theta_v^i.
\]
Now, define the event $A = \bigcap_{i \in [m]} \{\xi_{v^i}^i(n)
= 1\}$, with $\xi$ as defined by (\ref{eqn:xi}). That is, $A$ is the
event in which each walk $i$  makes a transition to vertex
$v^i$ at time $n+1$, $i = 1, 2, \ldots, m$. For all $n \geq 0$, we
have that for all $\theta \in \TsD_1$,
\begin{equation}\label{Edesig33}
 \E\Big[\big\langle\theta, U_n\big\rangle^+  \Big  | \, X(n) = x, \F_n
 \Big ] 
 =
 \E\Big[\big\langle\theta, U_n\big\rangle^+  \Big  | \, X(n) = x \Big
 ] 
 \geq
 q(x, \theta)
\end{equation}
where
\begin{equation} \label{eqn:qn}
 q(x, \theta)
 =
 \E\Big[\big\langle\theta, U_n\big\rangle^+  \Big  | \, A,\, X(n) = x 
 \Big]\Prb\big(A\, | \, X(n) = x\big).    
\end{equation}
 
To see that (\ref{Edesig33}) holds, note that the first equality
follows because the distribution of $U_n$ is uniquely
determined by $X(n)$ according to (\ref{eqn:gamma_and_U}).  The
inequality in (\ref{Edesig33}) holds because $\langle\theta,
U_n\rangle^+$ is non-negative. Now, to show (\ref{Edesig}), it is
sufficient to prove that for all $\theta \in \TsD_1$ and $x \in \D$
\begin{equation}\label{inetheta}
  q(x, \theta) \geq s(x).
\end{equation}

To show (\ref{inetheta}),  we show first that
\begin{equation}\label{eqn:defs-novo}
  q(\theta, x)
  =
  \Big [\sum_{i} \max_v \theta_{v}^i  - \sum_{i}
   \big\langle \theta^i, \pi^i  
   \big (x \big ) \big\rangle \Big ]^+  \prod_{i=1}^m
  \pi_{v^i}^{i}(x).
\end{equation}

To show (\ref{eqn:defs-novo}), note that, given $X(n) = x$, the
transitions of the walks are  independent, and therefore,   
\begin{equation}
\label{eqn:condprob}
\Prb\big(A\, | \, X(n) = x \big) = \prod_{i=1}^m \pi_{v^i}^{i} (x).
\end{equation}

To conclude the proof of (\ref{eqn:defs-novo}), we show that, given
$X(n) = x$ and $A$, we have that $\big\langle\theta, U_n\big\rangle =
\sum_{i} \max_v \theta_{v}^i - \sum_{i} \big\langle \theta^i, \pi^i
\big (x \big ) \big\rangle$.  According to (\ref{eqn:gamma_and_U}), we
have $(U_n)_v^i = \xi_v^i(n) - \E[\xi_v^i(n)\mid\F_n] = \xi_v^i(n) -
\pi_v^i(X(n))$, where, by (\ref{eqn:xi}), $\xi_v^i(n) = \Ind\{W^i(n+1)
= v\}$. Let $\delta_{v,v^i} = 1$ if $v = v^i$ and zero otherwise. So,
given $X(n) = x$ and $A$, it follows that $(U_n)_v^i = \delta_{v,v^i}
- \pi_v^i(x)$ and therefore
\begin{align*}
\big\langle\theta, U_n\big\rangle 
&=
\sum_{i,v} \theta_v^i \Big  (\delta_{v,v^i} - \pi_v^i \big
(x \big) \Big )                                              \\ 
&=
\sum_{i} \theta_{v^i}^i  - \sum_{i} \big\langle \theta^i, \pi^i 
\big (x \big ) \big\rangle  \\
&=
\sum_{i} \max_v \theta_{v}^i  - \sum_{i} \big\langle \theta^i, \pi^i 
\big (x \big ) \big\rangle. 
\end{align*}

Next we use (\ref{eqn:defs-novo}) to show (\ref{inetheta}). 
For $\theta^{i} \in \mathbb{R}^{d}$, we set  $(\theta^i)^+ =
((\theta_1^i)^+, (\theta_2^i)^+, \ldots,$ $(\theta_d^i)^+)$,  $(\theta^i)^- =
((\theta_1^i)^-, (\theta_2^i)^-, \ldots,$ $(\theta_d^i)^-)$,
and $\TsD_1 = \big \{ \theta \in \TsD  \,  : \, \sum_{iv} | \theta_v^i | = 1  \big \}$. 
To save notation, we set  $y = \pi(x)$. Now observe that
\begin{align*}
 q(\theta, x)
&=
 \Big [\sum_{i} \max_v \theta_{v}^i  - \sum_{i} \big\langle \theta^i,
 y^i \big\rangle \Big ]^+   \prod_{i=1}^m y_{v^i}^{i}                \\ 
&\geq
\Big [\sum_{i} \max_v \theta_{v}^i  - \sum_{i} \big\langle \theta^i, 
y^i \big\rangle \Big ]^+  \Big(\min_{i,v} y_v^i  \Big)^m            \\ 
&=
 \Big [\sum_{i} \max_v \theta_{v}^i  - \sum_{i} \big\langle
 (\theta^i)^+ - (\theta^i)^-,  
 y^i \big\rangle \Big ]^+  \Big(\min_{i,v} y_v^i  \Big)^m          \\ 
&=
 \Big [\sum_{i} \Big (\max_v \theta_{v}^i 
-  \big\langle (\theta^i)^+, 
y^i \big\rangle   \Big  ) 
+ \sum_{i} \big\langle (\theta^i)^-, 
y^i \big\rangle\Big ]^+  \Big(\min_{i,v} y_v^i  \Big)^m        \\ 
&\geq \Big [
 \sum_{i} \big\langle (\theta^i)^-, 
y^i \big\rangle\Big ]^+  \Big(\min_{i,v} y_v^i  \Big)^m        \\
&\geq \Big [
 \frac{1}{2 m d}
 \Big(\min_{i,v} y_{v}^i\Big) \Big ]^+  \Big(\min_{i,v} y_v^i  \Big)^m
  \\ 
&= 
\frac{1}{2 m d}  \Big(\min_{i,v} y_v^i  \Big)^{m+1}, 
\end{align*}
which shows (\ref{inetheta}) as claimed. Above, the first inequality
holds because $0 \leq \min_{i,v} y_v^i \leq y_{v^i}^i \leq 1$ for all
$i$.  The second inequality holds because $\max_v \theta_{v}^i \geq
(\theta^i)_v^+$ all $i$ and $v$, and because $y^i$ is a probability
measure for all $i$, and therefore, $\max_v \theta_{v}^i - \big\langle
(\theta^i)^+, y^i \big\rangle \geq 0$ for all $i$. To show the last
inequality, it is sufficient to show that
\begin{equation}\label{xthetaquota}
\sum_{i=1}^m \big\langle  (\theta^i)^-, y^i \big\rangle    
\geq 
\frac{1}{2 m d}\min_{i,v}\{y_{v}^i \}.
\end{equation}

To verify (\ref{xthetaquota}), observe that
\begin{align*}
  \sum_i\big\langle (\theta^i)^-, y^i \big\rangle
 =
   \sum_{i,v} y_v^i (\theta_v^i)^-
 &\geq
   \min_{i,v}\{y_{v}^i\} \sum_{i,v} (\theta_v^i)^-  \\
 &\geq
  \min_{i,v}\{y_{v}^i\} \max_{i,v} (\theta_v^i)^- \geq
  \min_{i,v}\{y_{v}^i\} \frac{1}{2 m d}.
\end{align*}

The last inequality is justified by observing that $\max_{i,v}
(\theta_v^i)^- \geq \frac{1}{2 m d}$. To check this, we show that $1
\leq 2 m d \max_{i,v} (\theta_v^i)^-$. Since $\theta \in \TsD_1$, it
follows that $1 = \sum_{i,v} |\theta_v^i|$ and therefore
\begin{align*}
1 = \sum_{i,v} |\theta_v^i| 
&= 
 \sum_{i= 1}^m \Big ( \sum_{v=1}^d (\theta_v^i)^+ + \sum_{v=1}^d
(\theta_v^i)^- \Big )            \\
&= 
\sum_{i= 1}^m  2 \Big(\sum_{v=1}^d (\theta_v^i)^- \Big) 
\leq 
2 m d \max_{i,v} (\theta_v^i)^-.
\qedhere
\end{align*}
\end{proof}

We will use the following definitions and lemma for the proof of the
first item in Theorem~\ref{thm11}.

\begin{definition}[Attractor]\label{defattractor} A subset
  $A\subset\mathfrak{D}$ is an \textit{attractor for the semi-flow
    $\Phi=\{\phi_t\}_{t\ge 0}$ if the following conditions hold:} 
\begin{enumerate}[(i), nosep]
\item $A$ is non-empty, compact and invariant by $\Phi$, that is,
  $\phi_t(A)=A, \forall  t\ge 0$;
\item  $A$ has a neighborhood $W\subset\mathfrak{D}$ such that
$\textrm{dist}(\phi_t(x),A)\to 0$ as $t\to\infty$ uniformly in $x\in
W$, 
\end{enumerate}  
where $\textrm{dist}(p,A)=\inf_{a\in A} \Vert p-a\Vert$. The
\emph{basin} of $A$, $B(A)$, is the positively invariant open set
formed by the points $x \in \D$ such that dist$(\phi_t(x), A) \to 0$
as $t \to \infty$.
\end{definition}

\begin{lemma}[{\cite[Theorem, (b), p. 181]{HS1974}}]\label{lattractor}
Let $x^*$ be a linearly stable equilibrium point of the vector field
$F$ defined in (\ref{vectorfieldF}). Then $A=\{x^*\}$ is an attractor
for the semi-flow $\Phi$ induced by $F$.
\end{lemma}

Let $\tau_0$ and $\tau_n = \sum_{k=1}^n \gamma_k$ for $n\geq 1$ with
$\gamma_k$ defined as in (\ref{eqn:gamma_and_U}). Let $Z = \{Z(t)\}$, $t
\in [0,\infty)$, be a continuous-time affine and piecewise constant
process defined by considering the linear interpolation of $X =
\{X(n)\}_{n \geq 0}$, that is
\begin{equation}
  \label{eqn:Z}
  Z(\tau_n + s) = X(n) + s \frac{X(n+1)-X(n)}{\tau_{n+1} - \tau_n},
  \qquad 0 \leq s \leq \gamma_{n+1},\quad n \geq 0.
\end{equation}

\begin{definition}\label{def:attainable} A point $x\in\D$ is said to
be \emph{attainable} by $Z = \{Z(t)\}$ if for each $t > 0$ and every
open neighborhood $U$ of $x$ 
\[
  \Prb\Big(\exists s \geq t : Z(s) \in U\Big) > 0.
\]
The set of attainable points of $Z$ is denoted by $\mathit{Att}(Z)$.
\end{definition}

\begin{proof}[Proof of Theorem \ref{thm11}]
Throughout let $X = \{X(n)\}_{n\geq 0}$ be the vertex occupation
measure process defined in (\ref{theX}) which satisfies the recursion
(\ref{eqn:SA}). Let $F$ be the smooth vector field defined in
(\ref{vectorfieldF}).

$(\mathit{i})$ Let $x^*$ be a linearly stable equilibrium of the
vector field $F$ and let $A = \{x^*\}$. The proof of the first
assertion follows from Theorem 7.3 in \cite{B99}, provided that
$\mathit{Att}(Z) \cap B(A) \neq \emptyset$, that is, provided the
basin of $A$ is attainable by the process $Z$ defined in
(\ref{eqn:Z}). It is sufficient to show that $\mathit{Att}(X) \cap
B(A) \neq \emptyset$ because $\lim_{n\to\infty} \gamma_n = 0$.  Here
$\mathit{Att}(X)$ refers to the set of points $x\in\D$ attainable by
$X$, that is, such that, for each open neighborhood $U$ of $x$ and
each $n_*\in\N$, we have that $\Prb(\exists n \geq n_*: X(n) \in U) >
0$ or, equivalently, $\exists n \geq n_*: \Prb(X(n) \in U) > 0$. Since
each equilibrium $x^*$ of $F$ is arbitrarily close to a rational point
$q \in \D$, it is sufficient to show that, for each such $q$, $n_*
\geq 0$ and $\varepsilon > 0$, there is a $n > n_*$ such that $\Prb
\big (|X(n) - q| < \varepsilon \big ) > 0$. To check this, let $q$
have components of the form $q_v^i = k_v^i/k$, where for $i \in [m]$,
$v \in [d]$, $k_v^i$ is a non-negative integer and $k$ is a positive
integer. Note that, since $\sum_v q_v^i = 1$, it follows that $\sum_v
k_v^i = k$ for all $i \in [m]$.  Now, consider the following sequence
of vertices $v^i(\bar n) \in [d]$, $i \in [m]$, $\bar n = 1, 2,
\ldots$, defined as follows. For each $i \in [m]$,  $v \in [d]$, and
$\bar n = 1, 2, \ldots$, we set 
\[
  v^i(\bar n) = v \quad \text{ if and only if }\quad \bar n \in
  \bigcup_{\ell =1}^\infty N_{\ell,v}^i 
\]
where, for each $i$, $v$, and $\ell = 1,2, \ldots$, the set
$N_{\ell,v}^i$ is defined as $N_{\ell,v}^i = \big\{n_{\ell,v}^i + 1,
n_{\ell,v}^i + 2,\cdots,n_{\ell,v}^i + k_v^i\big\}$, $n_{\ell,v}^i =
(\ell - 1)k + k_1^i + k_2^i + k_{v-1}^i$ for $v \geq 2$, and
$n_{\ell,v}^i = (\ell - 1) k$ for $v = 1$. In words, the sequence
$\{v^i(\bar n)\}_{\bar{n} \geq 1}$ is a cycling sequence of vertices
of $G$ for which the cycle, of length $k = k_1^i + k_2^i + \cdots +
k_{d}^i$, contains $k_v^i$ repetitions of vertex $v$.

Now, for each $n \geq 1$,  let $A_n$ be the event, in which the
process $W(\bar n)$ follows the vertex sequences $v(\bar n)$ up to
time $n$. That is
\[
  A_n = \bigcap_{\bar n = 1}^n \bigcap_{i = 1}^m  \Big\{W^i
  (\bar n) = v^i (\bar n) \Big\}.
\]

Choose $n = L k$, where $L$ is a sufficiently large integer, such that
$n > n_*$ and $n > m d/\varepsilon$. Given $A_n$, it holds that
$X_v^i(n) = (1 + L k_v^i) / (L k) = 1/ (L k) + q_v^i = 1/ n + q_v^i$,
in which case $|X(n) - q| = \sum_{i,v} |X_v^i(n) - q_v^i| = md/n <
\varepsilon$. Thus, assuming that $\Prb(A_n) > 0$, we have that
\[
  \Prb \big (|X(n) - q| < \varepsilon \big ) \geq \Prb \big (|X(n) -
  q| < \varepsilon \, \big | \,  A_n \big )\Prb \big (A_n \big ) =
  \Prb \big (A_n \big ) > 0.
\] 

It remains to show that $\Prb \big (A_n \big ) > 0$. Let $X_v^i(0) =
1$, and for $\bar n \in \{1,2, \ldots, n\}$, let $x_v^i(\bar n)$ be
the value of $X_v^i(\bar n)$ computed according to
(\ref{eqn:occupation_measure}) when $W^i(\bar n) = v^i(\bar n)$, $\bar
n = 1,2, \ldots, n$.  Since, for each $\bar
n \in \{1,2, \ldots, n\}$, we have that $W^i(\bar n)$, $i = 1,2,
\ldots, m$ are conditionally independent given $\{X(\bar n - 1) =
x(\bar n - 1)\}$, it follows, by (\ref{eqn:trans_prob}) and
(\ref{eqn:pi}) that
\begin{equation}
   \Prb \big (A_n \big ) = \prod_{\bar n = 1}^{n} \prod_{i=1}^m
   \pi_{v^i(\bar n)}^i(x(\bar n - 1)).
\end{equation}
Since $\pi_v^i(y) > 0$ for all $i, v$ and $y \in \D$, it follows that
$\Prb \big (A_n \big ) > 0$. This concludes the proof of the first
claim made in the theorem.

$(\mathit{ii})$ The proof of the second claim follows by Theorem~1 in
\cite{P90}. In order to use this result, we observe that all the
required assumptions and hypotheses required by this theorem are
easily verified for the vertex occupation measure process
$X=\{X_n\}_{n\geq 0}$ satisfying (\ref{eqn:SA}). Only the condition
determined by (\ref{eqn:NoiseboundBellow}) is more involved and
deserves special attention. This condition is satisfied in our case by
Lemma~\ref{lem:boundfrombellow}.

$(\mathit{iii})$ By Lemma \ref{lm:LimitSet}, the limit set of
$X=\{X(n)\}_{n\ge 0}$ is connected and contained in $\CR$, where
$\Phi=\{\phi_t\}_{t\ge 0}$ is the semi-flow induced by the vector
field $F$. Moreover, since $F$ is a continuous map on its compact
domain $\D$ with isolated zeros (equilibrium points), we have that $F$
has finitely many equilibrium points. Hence, if $L$ is the strict
Lyapunov function defined in Theorem \ref{th:Lyapunov}, then
$L\big(F^{-1}(0)\big)$ is a finite set. In this case, by Proposition
3.2 in \cite{B96}, it follows that $\CR$ is contained in the set of
equilibrium points. Since $\sL\big(\{X(n)\}\big)$ is connected, we
have proved that $\sL\big(\{X(n)\}\big)$ is an equilibrium point of
$F$ that may depend on $\omega$. However, because $\omega$ is an
arbitrary point of $\Omega$, we have that
\[
  \sum_{x\in F^{-1}(0)}\mathbb{P}\big(\lim_{n\to\infty}
  X(n)=x\big)=1.
\]
This concludes the proof of the theorem.
\end{proof}

\section{Proofs of Theorems~\ref{thm22}, \ref{thm33},
  \ref{th:TwoRWonK2}, and \ref{th:3RW}}
\subsection{Theorems~\ref{thm22} and \ref{thm33}}
In order to prepare for the proof of Theorems \ref{thm22} and
\ref{thm33}, we will provide conditions on the constants
$\alpha_u^{ij}\in\mathbb{R}$ under which the map $\pi$ defined in
\eqref{eqn:pi} has a unique fixed point or, equivalently, the vector
field $F$ defined in \eqref{vectorfieldF} has an unique equilibrium
point.

We will need the following injectivity result.

\begin{theorem}[Gale-Nikaid\^o \cite{DGHN1965}]\label{thmgn}
Let $n\ge 2$, $\Lambda\subset\mathbb{R}^n$ be an (open or closed)
convex set and $f:\Lambda\to\mathbb{R}^n$ be a differentiable map
whose Jacobian matrix $\JacF(x)$ is positive quasi-definite for all
$x\in\Lambda$. Then $F$ is injective on $\Lambda$.
\end{theorem}

To be more precise, the statement above is Theorem $6$ of
\cite{DGHN1965} (see also Theorem 3 of \cite{TP1983}). We recall that
an $n\times n$ matrix $A=(a_{ij})$ with real entries is
\textit{positive quasi-definite} if its symmetric part, namely
$\frac12(A+A^T)$ is positive-definite, i.e. $x^T (A+A^T) x>0$ for all
non-null $n\times 1$ column matrix $x$.

\begin{proof}[Proof of Theorem \ref{thm33}] Since $\pi$ is a
continuous self-map of the compact convex set $\D$, by Brouwer's Fixed
Point Theorem, $\pi$ has at least one fixed point $x^*$ in
$\D$. Moreover, $x\in\D$ is a fixed point of $\pi$ if and only if $x$
is a zero of $F:\D\to \mathbb{R}^{dm}$ defined in
(\ref{vectorfieldF}). Hence, to prove that $\pi$ has a unique fixed
point in $\D$ it suffices showing that $F$ (as a map) is injective on
$\D$.

Let $x=(x_u^j)\in\D$. By (\ref{vectorfieldF}), the Jacobian of $F$ at
$x$ is the $dm\times dm$ matrix 
\[
\dfrac{\partial F_v^i}{\partial x_u^j}(x)=\begin{cases} -1 &
  \textrm{if}\quad u=v \quad\textrm{and}\quad j=i, \\[0.1in]    
\phantom{-}\alpha_{v}^{ij}\dfrac{\exp\big(\sum_{j=1}^m \alpha_v^{ij}
  x_v^j\big)\sum_{\substack{u=1\\u\neq v}}^d 
   \exp\big(\sum_{j=1}^m \alpha_u^{ij} x_u^j\big)
} {\left[\sum_{u=1}^d
   \exp\big(\sum_{j=1}^m \alpha_u^{ij} x_u^j\big)\right]^2} &
\textrm{if}\quad u=v \quad\textrm{and}\quad j\neq i, \\[0.3in]
-\alpha_{u}^{ij}\dfrac{\exp\big(\sum_{j=1}^m \alpha_v^{ij}
  x_v^j\big)\exp\big(\sum_{j=1}^m \alpha_u^{ij} x_u^j\big)
} {\left[\sum_{u=1}^d
   \exp\big(\sum_{j=1}^m \alpha_u^{ij} x_u^j\big)\right]^2}
& \textrm{if}\quad u\neq v \quad\textrm{and}\quad j\neq i, \\
\phantom{-} 0 & \textrm{if}\quad u\neq v \quad\textrm{and}\quad j=i.
\end{cases}
\]
where we used the hypothesis that $\alpha_u^{ij}=0$ if $i=j$.

In this way, we have that  
\begin{equation}\label{dFdx}
\frac{\partial F_v^i}{\partial x_u^j}(x)=\begin{cases} -1 &
  \textrm{if}\quad u=v \quad\textrm{and}\quad j=i, \\   
\phantom{-}\alpha_{v}^{ij} \pi_v^i(x) (1- \pi_v^i (x)) &
\textrm{if}\quad u=v \quad\textrm{and}\quad j\neq i, \\
-\alpha_{u}^{ij} \pi_v^i(x) \pi_u^i(x) 

& \textrm{if}\quad u\neq v \quad\textrm{and}\quad j\neq i,\\
\phantom{-} 0 & \textrm{if}\quad u\neq v \quad\textrm{and}\quad j= i. 
\end{cases}
\end{equation}

Since $\sum_{u=1}^d \pi_u^i(x)=1$, we have that for all $u\neq v$,
\[
  \pi_v^i(x)\pi_u^i (x)= \pi_v^i(x) \Bigg(1- \sum_{r\neq
    u}\pi_r(x)\Bigg)\le \pi_v^i(x) (1-\pi_v^i(x))\le \frac14.
\]
Putting it all together, we have that for each $(v,i) \in
[d]\times[m]$,
\[
  \sum_{(u,j)\neq (v,i)} \left| \frac{\partial F_v^i}{\partial
      x_u^j}(x)\right|\le  
  \sum_{\substack{j=1\\ j\ne i}}^m \frac14\left| \alpha_v^{ij}\right|
  +  \sum_{\substack{u=1\\ u\neq v}}^d \sum_{\substack{j=1\\ j\ne
      i}}^m \frac14 \left| \alpha_u^{ij}  \right\vert= \sum_{u=1}^d
  \sum_{\substack{j=1\\ j\ne i}}^m \frac14 \left| \alpha_u^{ij}
  \right\vert.
\]
Hence, if Condition (C1) holds, then  
\[
  \sum_{(u,j)\neq (v,i)} \left| \frac{\partial F_v^i}{\partial
      x_u^j}(x)\right|\le \frac14\sum_{u=1}^d \sum_{\substack{j=1\\
      j\ne i}}^m  \left| \alpha_u^{ij} \right\vert<1= 
 \left| \frac{\partial F_v^i}{\partial x_v^i}(x)\right|.
\]
This shows that the Jacobian matrix $\JacF(x)=\big((\partial
F_v^i/\partial x_u^j)(x)\big)$ is strictly row diagonally dominant.
By \eqref{eqn:strengths}, $\JacF(x)$ is also symmetric. Now we use a
result from Linear Algebra that asserts that every symmetric strictly
(row or column) diagonally dominant matrix with real entries and
positive diagonal entries is positive-definite. By Theorem
\ref{thmgn}, we have that $F$ is injective on $\mathfrak{D}$. We have
proved that $\pi$ has a unique fixed point in $\D$.  The proof is
concluded by applying Theorem \ref{thm11}.
\end{proof}

\begin{proof}[Proof of Theorem \ref{thm22}] Assume first that
Condition (C2) is true. By definition, the constants
$\alpha_v^{ij},$\, $v \in [d]$ and $i, j \in [m]$ satisfy
\eqref{eqn:strengths}. We claim that Condition (C3) is
true. In fact, $\alpha_v^{ii}=0$ for all $v,i$, and
\[
  \sum_{u=1}^d \sum_{\substack{j=1\\ j\ne i}}^m  \left| \alpha_u^{ij}
  \right\vert =d(m-1)\beta<4 
  \quad\textrm{for each}\quad (v,i)\in [d]\times[m]. 
\]
By Theorem~\ref{thm33}, the process $X(n) = (X^1(n), \ldots, X^m(n))$
converges almost surely to the unique equilibrium point of $F$. 

Now let us consider the case in which Condition (C1), rather than
Condition (C2), is true. Since $d=2$, we have that $u, v \in \{1,2\}$
and $x^i_u = 1- x_v^j$ for all $u \neq v$ and $j \in [m]$. In this
way, the map $\pi$ simplifies into
 \begin{align*}\label{eqn:p2}
   \pi^i_v(x)
   &= 
   \frac{\exp\big(\sum_{j\neq i} -\beta x_v^j\big)} {
   \exp\big(\sum_{j\neq i} -\beta x_v^j\big)+\exp\big(\sum_{j\neq i}
   -\beta (1-x_{v}^j)\big)}            \\
   &=
   \frac{1}{1+\exp\big(\sum_{j\neq i} -\beta
   (1-2x_{v}^j)\big)} 
\end{align*}
Hence,
\[
  \pi^i_v(x) = \psi\bigg(\sum_{j\ne i} x_v^j\bigg),
  \quad\textrm{where}\quad
  \psi(t)=\frac{1}{1+\exp\big(-\beta(m-1-2t)\big)}\cdot 
\] 
Hence, $x=(x_v^i)$ is a fixed point of $\pi$ if and only if
\begin{equation}\label{systemI}
   x_v^i = \psi\bigg(\sum_{j\ne i} x_v^j\bigg), \quad i=1,2,\ldots,m;
\quad v=1,2.
\end{equation}
The function $\psi $ is monotone, hence invertible. Therefore,
\eqref{systemI} is equivalent to
\begin{equation}\label{systemII}
 \psi^{-1}(x_v^i) = \sum_{j\ne i} x_v^j, \quad i=1,2,\ldots,m; \quad
 v=1,2. 
\end{equation}
In this way, for all $1\le i,k\le m$ with $i\neq k$ we have that
\[
  \psi^{-1}(x_v^i)-\psi^{-1}(x_v^k)=x_v^k-x_v^i.
\]
That is to say,
\[ 
  \psi^{-1}(x_v^i) + x_v^i=\psi^{-1}(x_v^k) + x_v^k.
\]
Defining $\varphi:\mathbb{R}\to\mathbb{R}$ as $\varphi(t)=\psi(t)+t$ 
leads to 
\[
  \varphi(x_v^i)=\varphi(x_v^k).
\]
Since $\beta \leq 2$, we have that
\[
  \varphi'(t)
  =
  \psi'(t)+1
  =-2\beta\frac{\exp\big(-\beta(m - 1 -
    2t)\big)}{\left[1+\exp\big(-\beta(m - 1 -2t)\big)\right]^2}+1>0, 
\]
implying that $\varphi$ is monotone, hence injective. Therefore,
\[
  x_v^i = x_v^k, \quad i\neq k, \quad v=1,2.
\]
We have proved that
\begin{equation}\label{equality}
  x_v^1=x_v^2=\ldots =x_v^m,\quad v=1,2.
\end{equation}
We claim that $x_v^i=\frac12$ for all $v,i$. By way of contradiction,
without loss of generality, suppose that $x_1^1>\frac12$. Then, by
\eqref{equality}, we have that $ \sum_{j>1} x_1^j >
\frac{m-1}{2}$. Replacing this in \eqref{systemI} and using the fact
that $\psi$ is decreasing gives
\[
   \frac12< x_1^1= \psi\bigg(\sum_{j>1}
     x_1^j\bigg)<\psi\Big(\frac{m-1}{2}\Big)=\frac12, 
\]
which is a contradiction. This shows that
$x^*=\big(\frac12,\frac12,\ldots,\frac12\big)$ is the unique fixed
point of $\pi$ and the unique equilibrium of $F$. The application of
Theorem~\ref{thm11} concludes the proof.
\end{proof}

\subsection{The proof of Theorem~\ref{th:3RW}}
We present first a couple of lemmas that will be used in
the proof of Theorem~\ref{th:3RW} and then conclude its
proof. Throughout this section, we will use $p \in \D$ to denote the
point
\[
   p = \big(\textstyle\frac12, \frac12, \frac12, \frac12, \frac12,
   \frac12\big). 
\]

\begin{lemma}\label{lem:eigenvalues}
Let $X=\{X(n)\}_{n\ge 0}$ be the process defined in
(\ref{theX}) and $F$ the vector field defined in
(\ref{vectorfieldF}). There is $\beta_0>2 $ such that for all
$\beta\ge \beta_0$ the following statements hold
\begin{enumerate}[(i), nosep]
\item There exists a unique $w_\beta\in
  \big(0,\frac{1}{\beta^3}\big)$  such that
  $w_\beta=1/\big(1+e^{2\beta(\frac12-w_\beta)}\big)$; 
\item The set $S$ defined as
\[
    S=\left\{\big(a,1-a,b,1-b,c,1-c\big):\{a,b,c\}
      =\left\{\frac12,w_\beta,1-w_\beta\right\}, a\neq b\neq c \right\}
\]
consists of linearly stable equilibrium points of the vector field $F$ 
\item $\mathbb{P}\big( \lim_{n\to\infty} X(n)=x\big)>0$ for each $x\in
 S$.
\end{enumerate}
\end{lemma}
\begin{proof} $(i)$ Since
  \[
    \lim_{\beta\to\infty}
    \frac{\beta^3}{1+e^{\left(\beta-\frac{2}{\beta^2}  \right)}}=0,
\]
there exists $\beta_1>2$ be so large that
\begin{equation}\label{beq1}
\frac{1}{1+e^{\left(\beta-\frac{2}{\beta^2} \right)}} <
\frac{1}{2\beta^3}. 
\end{equation}
In what follows, we assume that $\beta\ge \beta_1$.

Let $\varphi:\mathbb{R}\to\mathbb{R}$ be the function defined by
$\varphi(t)=1/\big(1+e^{2\beta (\frac12-t)}\big)-t$. By \eqref{beq1},
we have that
\[
  \varphi(0)>0\quad\textrm{and}\quad
  \varphi\left(\frac{1}{\beta^3}\right)
  =
  \frac{1}{1+e^{(\beta-\frac{2}{\beta^2})}}-\frac{1}{\beta^3}<0.
\]
In this way, by the continuity of $\varphi$, we have that $\varphi$
has a zero in the interval $\big(0,\frac{1}{\beta^3}\big)$.  Using
\eqref{beq1} once more, we have that
\[
  \varphi'(t)=2\beta
  \frac{e^{2\beta\big(\frac12-t\big)}}{1+e^{2\beta\big(\frac12-t\big)}}
  \cdot\frac{1}{1+e^{2\beta\big(\frac12-t\big)}}-1\le
  \frac{2\beta}{1+e^{(\beta-\frac{2}{\beta})}}-1<0,\quad\forall t\in
  \big(0,\frac{1}{\beta^3}\big).
\] 
Hence, $\varphi$ is strictly decreasing on
$\big(0,\frac{1}{\beta^3}\big)$. Therefore, there exists a unique
$w_\beta\in \big(0,\frac{1}{\beta^3}\big)$ such that
$\varphi(w_\beta)=0$, i.e.,
$w_\beta=1/\big(1+e^{2\beta(\frac12-w_\beta)}\big)$.

$(\mathit{ii})$ We will prove now that every point of $S$ is an
equilibrium point of $F$.  By the definition of the vector field $F$,
we have that $x=(a,1-a,b,1-b,c,1-c)$ is an equilibrium point of $F$ if 
\[
  a=\psi\big(b+c\big), \quad
  b=\psi(a+c),\quad\textrm{and}\quad c=\psi(a+b),
\]
where $\psi:[0,2]\to\mathbb{R}$ is defined by
\[
  \psi(t)=\frac{1}{1+e^{2\beta (t-1)}}.
\]
Since $x\in S$, we have that $a+b+c=\frac32$. Hence, the following
conditions are sufficient for $x\in S$ to be an equilibrium point of
$F$:
\[
  a=\psi\Big(\frac32-a\Big)=\varphi(a),\quad
  b=\psi\Big(\frac32-b\Big)=\varphi(b),\quad
  c=\psi\Big(\frac32-c\Big)=\varphi(c),
\]
where $\varphi$ is the function used in the definition of $w_{\beta}$. 

In other words, $x=(a,1-a,b,1-b,c,1-c)$ is an equilibrium point of $F$
if and only if each value $u\in
\{a,b,c\}=\left\{\frac12,w_\beta,1-w_\beta\right\}$ is a fixed point
of $\varphi$. It is easy to verify that $u=\frac12$ is a fixed point
of $\varphi$. Moreover, $w_\beta$ is a fixed point of $\varphi$ by
definition. Finally, $u=1-w_\beta$ is a fixed point of $\varphi$
because
\[
  \varphi(1-w_\beta)
  =
  \frac{1}{1+e^{-2\beta\big(\frac12-w_\beta\big)}}
  =
  \frac{e^{2\beta\big(\frac12-w_\beta\big)}}{ 1+
    e^{2\beta\big(\frac12-w_\beta\big)}}
  =
  1-\frac{1}{1+e^{2\beta\big(\frac12 -
    w_\beta\big)}}1-\varphi(w_\beta)=1-w_\beta.
\]
We have proved that every point in $S$ is a fixed point of $F$.

It remains to prove that if $x=x(\beta)=(a,1-a,b,1-b,c,1-c)\in S$,
then $x$ is linearly stable for $\beta$ big enough. In fact, the
Jacobian matrix of $F$ at $x$ is given by
\[
  \begin{bmatrix}
    -1 & 0 & \psi'(b+c) & 0 &\psi'(b+c) & 0 \\
    0 & -1 & 0 &  \psi'(2-b-c)  & 0  &  \psi'(2-b-c) \\
    \psi'(a+c)  & 0 & -1 & 0 & \psi'(a+c)  & 0 \\
    0 &  \psi'(2-a-c)  & 0 & -1 & 0 & \psi'(2-a-c)   \\
    \psi'(a+b)  & 0 & \psi'(a+b) & 0 & -1 & 0 \\
    0 &  \psi'(2-a-b)  & 0 & \psi'(2-a-b) & 0 & -1 \\
  \end{bmatrix}.
\]
Using the elementary facts that $\psi'(t)=-2\beta \psi(t)
(1-\psi(t))$, $\psi(2-t)=1-\psi(t)$, $a=\psi(b+c)$, $b=\psi(a+c)$,
$c=\psi(a+b)$ and further defining
\begin{equation}\label{abc}
  \undd{a}=-2\beta a (1-a), \quad \undd{b}=-2\beta b (1-b),\quad
  \undd{c}=-2\beta c (1-c),
\end{equation}
the Jacobian can be written as
\[
  \begin{bmatrix}
    -1 & 0 & \undd{a} & 0 & \undd{a} & 0 \\
    0 & -1 & 0 &  \undd{a}  & 0  &  \undd{a}  \\
    \undd{b} & 0 & -1 & 0 & \undd{b}  & 0 \\
    0 &  \undd{b} & 0 & -1 & 0 & \undd{b}   \\
    \undd{c} & 0 & \undd{c} & 0 & -1 & 0 \\
    0 & \undd{c} & 0 & \undd{c} & 0 & -1 \\
  \end{bmatrix}.
\]
The characteristic polynomial of the Jacobian matrix is therefore
\begin{equation}\label{pcha}
  p_{x}(\lambda)=\left(\undd{a} \undd{b} + \undd{a}\undd{c} + 2
    \undd{a} \undd{b}\undd{c} + 
\undd{a} \undd{b} \lambda + \undd{a}\undd{c}\lambda -
(1+\lambda)(-\undd{b} \undd{c} + (1+\lambda)^2)\right)^2. 
\end{equation}
Since $\{a,b,c\}=\left\{\frac12,w_\beta,1-w_\beta\right\}$,
$w_\beta\in\left(0,\frac{1}{\beta^3}\right)$ and $1-w_\beta\in (0,1)$, 
we have that
\[
  \undd{a}\undd{b}<4\beta^2 \max \left\{
   \frac14,\frac{1}{\beta^3}\right\} \frac{1}{\beta^3}\le
 \frac{1}{\beta}.
\]
Likewise, we have that
\begin{equation}\label{abc4}
  \undd{a}\undd{b}<\frac{1}{\beta},\quad
  \undd{a}\undd{c}<\frac{1}{\beta}, \quad
  \undd{b}\undd{c}<\frac{1}{\beta}\quad\textrm{and}\quad
  \undd{a}\undd{b}\undd{c}<\frac{1}{\beta}.
\end{equation}
Each equilibrium point $x=x(\beta)=(a,1-a,b,1-b,c,1-c)\in S$ depends
on $\beta$. In particular, when $\beta\to\infty$, we obtain by
(\ref{pcha}) and (\ref{abc4}) that
\[
 \lim_{\beta\to\infty} p_{x(\beta)}(\lambda) =  (1+\lambda)^6.
\]
Hence, since the entries of $\mathit{JF}\big(x(\beta)\big)$ depend
smoothly on $\beta$, we conclude that if $\beta$ is big enough, say
$\beta\ge\beta_0$, then all the eigenvalues of
$\mathit{JF}\big(x(\beta)\big)$  will lie in an open ball centered at
$-1\in\mathbb{C}$ of radius $\frac12$. Therefore, they all will have
negative real parts, that is, $x(\beta)\in S$ will be a linearly
stable equilibrium point.
 
(iii) This follows from the first item in Theorem \ref{thm11}. 
\end{proof}

\begin{lemma}\label{a:unstable} If $\beta < 2$, then $p$ is linearly
stable. If $\beta > 2$, then $p$ is linearly unstable. 
\end{lemma}
\begin{proof}
The proof consists in studying $\sigma(\JacF(p))$, the set of
eigenvalues of the Jacobian matrix of $F$ at $p$. Relatively simple
calculations show that the characteristic polynomial of $\JacF(p)$
equals
\[
  (1+\lambda)^3 \Big(-1+\frac{\beta}{2} -\lambda\Big)^{2}
  (1+\beta + \lambda).
\]
The equilibrium $p$ is therefore hyperbolic. Further, up to
algebraic multiplicity, the eigenvalues of $\JacF(p)$ are
\[
 -1,\quad -1-\beta \quad\text{and}\quad -1+\frac{\beta}{2}.
\] 
This shows that $p$ is linearly stable when $\beta<2$ and linearly
unstable when $\beta > 2$.
\end{proof}

\begin{proof}[Proof of Theorem~\ref{th:3RW}]
 Let $W = \{W(n)\}_{n\geq 0}$ with $W(n) = (W^1(n)$, $W^2(n)$,
$W^3(n))$ be the process defined by $m = 3$ interacting random walks,
taking values on the complete graph $G$ with vertices $V = \{1,2\}$,
such that for all $n\ge 1$ and any $i \in [m]$,
\[
\begin{array}{l}
  \big\{W^i(n) = 1\big\} = \big\{S_{n}^i - S_{n-1}^i = -1\big\}, 
  \\[.35em]
  \big\{W^i(n) = 2\big\} = \big\{S_{n}^i - S_{n-1}^i = +1\big\}.
\end{array}
\]

Now, let $X_v^i(0) = 1$ for all $v \in [d] = \{1,2\}$ and $i \in [m]$,
and then, for $n \geq 1$ define $X_v^i(n)$ as in
(\ref{eqn:occupation_measure}).  Notice that $X_1^i(n)$ and $X_2^i(n)$
are the proportions of times the $i$-th walk, that is $S^i(n)$, makes
a transition to the left and to the right, respectively. Finally, for
$v \in [d]$ and $i,j \in [m]$, set
\[
  \alpha_v^{ij}
  =
  \begin{cases}
    -\beta, \ &\text{ if }\ \ i \neq j,\\
    0,  \ &\text{ if }\ \ i = j
  \end{cases}
\]
and let $\beta \geq 0$.  Using (\ref{eqn:transProb}) and (\ref{eqn:psi}),
it is readily seen that the transition probability for $W^i$ is given
by (\ref{eqn:trans_prob}). Indeed, since $X^i_1(n) = 1- X^i_2(n)$,
\begin{align*}
\label{eqn:trans_prob-Sn}
  \Prb\big(W^i(n+1) =  2 \mid\F_n\big)
  &=
    \Prb\big( S_{n}^i -  S_{n-1}^i = +1 \mid 
    \mathcal{A}_n\big) \\  
  &=
    \mu\Big((S_{n}^j - S_{0}^j)/n + (S_{n}^k - S_{0}^k)/n\Big) \\ 
  &=
    \mu\Big(2X_2^j(n) -1 + 2X_2^k(n) -1 \Big) \\
  &=
   \frac{\exp \Big(-\beta(X_2^{j}(n) + X_2^k(n))\Big
    )}{\sum_{v=1}^2 \exp\Big(-\beta(X_v^{j}(n)+X_v^k(n))\Big)}
 = 
    \pi_2^i(X(n)).
\end{align*}

Observe that Lemma~\ref{lem:XSA} holds in this case. That is, $X$ is a
stochastic approximation with $F(x) = - x + \pi(x)$ and with $\xi$, $U_n$,
and $\gamma_n$ given as in Lemma~\ref{lem:XSA}.

To prove the first assertion of the theorem, let $\beta < 2$. By
Theorem~\ref{thm22}, $p$ is the only equibrium, which by
Lemma~\ref{a:unstable} is linearly stable.  Using item
($\mathit{iii}$) of Theorem~\ref{thm11} we have that $X(n) \to p$
almost surely. As a consequence, almost surely it holds that
\begin{align*}
   \lim_{n\to\infty}
   \Prb\big( S_{n}^i - S_{n-1}^i = +1 \mid \mathcal{A}_n\big)
   &=
   \lim_{n\to\infty}
   \Prb\big(W^i(n+1) =  2 \mid\F_n\big)      \\
   &=
   \lim_{n\to\infty} \pi_2^i(X(n)) = \pi_2^i(p) = \frac12
\end{align*}
where the last two equalities follow by continuity of $\pi$ and the
fact that $p = (\frac{1}{2}, \ldots, \frac{1}{2})$ is a fixed point of
$\pi$. This concludes the proof of the first part.

The second assertion of the theorem is proved as follows. For
sufficiently large $\beta$, namely when $\beta \geq \beta_0 > 2$,
Lemma~\ref{a:unstable} and Theorem~\ref{thm11}.(ii) rule out the
possibility of converging to $p$. Further, by the proof of
Lemma~\ref{lem:eigenvalues}, there is a $w > \frac12$, such that
$x=(w, 1-w, 1-w, w, \frac12, \frac12)$ is a linearly stable
equilibrium. By item (i) of Theorem~\ref{thm11}, it follows that $X(n)
\to x$ with positive probability. As shown previously we have that
$\Prb(W^i(n+1) = v | \F_n ) = \pi_v^i(X(n))$. Again, by continuity of
$\pi$ and using the fact that $x$ is a fixed point of $\pi$, we have,
with positive probability, that
\begin{align*}
   \lim_{n\to\infty}
   \Prb\big( S_{n}^1 - S_{n-1}^1 = +1 \mid \mathcal{A}_n\big)
   &=
   \lim_{n\to\infty}
   \Prb\big(W^1(n+1) =  2 \mid\F_n\big)      \\
   &=
   \lim_{n\to\infty} \pi_2^1(X(n)) = \pi_2^1(x) = 1-w < \frac12
\end{align*}
Likewise,
\[
  \lim_{n\to\infty} \Prb\big( S_{n}^2 - S_{n-1}^2 = +1 \mid
  \mathcal{A}_n\big) = w > \frac12  
  \quad\text{and}\quad 
  \lim_{n\to\infty} \Prb\big( S_{n}^3 - S_{n-1}^3 = +1 \mid
  \mathcal{A}_n\big) = \frac12. 
\]
This concludes the proof.
\end{proof}

\subsection{The proof of Theorem~\ref{th:TwoRWonK2}}
The following lemma will be used for the proof of
Theorem~\ref{th:TwoRWonK2}. This lemma shows that the set of
equilibria for the example of two repelling walks on the two vertex
graph presented in Section~\ref{subsec:TwoRMonK2}, is finite for all
$\beta \geq 0$. This lemma also identifies the form of the equilibria
and characterises their stability.

\begin{lemma}\label{lem:TwoElephants_Equil}
Let $\Phi$ be the semi-flow induced by the ODE (\ref{eqn:ODE}) with $F$
given by (\ref{eqn:ODEexplicit}). Then,
\begin{enumerate}[(i), nosep]
\item The point $(\frac{1}{2},\frac{1}{2}, \frac{1}{2},
\frac{1}{2})$ is the only equilibrium of $\Phi$ when $\beta \in
[0,2]$.
\item When $\beta>2$, there exist two further equilibria of the form
\[
 (a, 1-a, 1-a, a)\quad \text{ and }\quad (1-a, a, a, 1-a),
\]
where $a \in (0, \frac{1}{2})$ is uniquely determined by $\beta$. 
\item $(\frac{1}{2},\frac{1}{2}, \frac{1}{2},\frac{1}{2})$ is linearly
stable when $\beta \in [0,2)$ and linearly unstable when $\beta >
2$. The equilibria in (ii) are linearly stable when $\beta > 2$.
\end{enumerate}
\end{lemma}
\begin{proof}
See Lemma 8 in \cite{CPR20}.
\end{proof}

\begin{proof}[Proof of Theorem~\ref{th:TwoRWonK2}] From
Lemma~\ref{lem:TwoElephants_Equil} it follows that $\Eq$ is formed by
isolated points. The proof is therefore concluded by direct
application of item ($\mathit{iii}$) in Theorem~\ref{thm11}.
\end{proof}

\begin{remark}\label{rem:BH}
The convergence of $X=\{X(n)\}_{n \geq 0}$ towards $\Eq$ can be
established in this example without using Theorems~\ref{th:Lyapunov}
and \ref{thm11}. An alternative proof is obtained from Theorem 6.12,
Corollary 6.13 and Theorem 6.15 in \cite{BH99}; see also Theorem 3.2
and Corollary 3.3 in \cite{B99}. In order to be able to use these
results, the semi-flow $\Phi$ defined by vector field determined by
the interacting random walks has to be planar. This is indeed the case
in this example. To observe this, it suffices to identify $\D$ with
$[0,1]^2$ by using the map $\eta: (a, 1-a) \mapsto a$ for $a \in
[0,1]$, and then consider the projection of the field on $[0,1]^2$
given by $F=(F^1_1, F^2_1)$. These steps cannot be carried out in the
examples of Section~\ref{subsec:weakRW} nor in the example presented
in Section~\ref{subsec:3RW}. More generally, the arguments presented
in the mentioned literature cannot be used when $m\geq 3$. The main
reason is that the dynamics induced by three or more interacting
random walks cannot be identified with a subset of the plane. Indeed,
when $m \geq 3$, the domain $\D$ may be identified via $\eta$ with the
$m$ dimensional unit-cube $[0, 1]^m \subset \R^m$.
\end{remark}

\bibliographystyle{authordate1}

\vspace{1cm}
\end{document}